\documentclass[11pt]{article}

\usepackage{amsmath,bbold,amssymb,graphicx,amsthm}
\usepackage{enumerate,color}

\newcommand{\grad}{\operatorname{grad}}

\renewcommand{\div}{\operatorname{div}}
\newcommand{\re}{\operatorname{Re}}
\newcommand{\im}{\operatorname{Im}}

\newtheorem{theorem}{Theorem}[section]
\newtheorem{lemma}[theorem]{Lemma}
\newtheorem{proposition}[theorem]{Proposition}

\newtheorem{remark}[theorem]{Remark}

\newtheorem{problem}[theorem]{Problem}
\newtheorem*{assumptiona}{Assumption A}
\newtheorem*{assumptionb}{Assumption B}
\newtheorem*{hypothesisG}{Hypothesis G}
\newtheorem*{workflow}{Workflow of the method}

\newcommand{\dist}{\operatorname{dist}}
\newcommand{\T}{\mathsf{T}}
\renewcommand{\H}{\mathsf{H}}

\usepackage[final]{showkeys}
\newcommand{\pconst}{\kappa_{\mathrm{g}}}

\renewcommand{\S}{Section~}

\title{Computation of sharp estimates of the Poincar{\'e} 
constant on planar domains with piecewise self-similar boundary}

\author{{\bf Lehel Banjai$^\mathrm{a}$ \qquad \& \qquad  Lyonell Boulton$^\mathrm{b}$} \\ \ \\
$^\mathrm{a,b}$Department of Mathematics \\
Maxwell Institute for Mathematical Sciences\\
Heriot-Watt University, Edinburgh EH14 4AS, Scotland.
\\ \\
$^{\mathrm{b}}$Department of Mathematics \\ Faculty of Nuclear Sciences and Physical Engineering \\
Czech Technical University in Prague \\ 
Trojanova~13, 12000 Prague 2, Czech Republic. \\ \\
}

\date{8th April 2019}

\begin{document}
\maketitle
\begin{abstract}
We establish a strategy for finding sharp upper and lower numerical bounds of the Poincar\'{e} constant on a class of planar domains with piecewise self-similar boundary. The approach consists of four main components: W1) tight inner-outer shape interpolation, W2) conformal mapping of the approximate polygonal regions, W3) grad-div system formulation of the spectral problem and W4) computation of the eigenvalue bounds. After describing the method, justifying its validity and determining general convergence estimates, we show concrete evidence of its effectiveness by computing lower and upper bound estimates for the constant on the Koch snowflake.
  \end{abstract}



\pagebreak

\section{Introduction} \label{introduction}

The Poincar{\'e}  constant $\pconst>0$ of a planar open set $\Sigma$ is the smallest $\kappa>0$ for which
\[
   \int_{\Sigma} |u|^2 \leq \kappa \int_{\Sigma}|\grad u|^2 \qquad \qquad \forall u\in \mathsf{H}^1_0(\Sigma).
   \]
Namely, $\omega_1^2=\frac{1}{\pconst}$ is the ground eigenvalue of the Dirichlet Laplacian on $\Sigma$. When the boundary $\partial\Sigma$ is a fractal curve, finding accurate estimates for this constant is highly non-trivial.  Classically, much of the important work in this area \cite{1995Lapidus,1997Pang} has focused on determining asymptotics for $\omega_1^2$ in terms of inner approximations of $\Sigma$, with the notable exception of a few numerical results reported in the literature for the particular case of the Koch snowflake \cite{1996Lapidus,1997Lapidus,2006Neuberger,2007Banjai}. 

In this paper we describe a method for computing tight upper and lower approximations of $\pconst$, when there exist two sequences of simply connected open polygons, $\{\T_j\}_{j=0}^\infty$ and $\{\H_j\}_{j=0}^\infty$, such that
\begin{equation} \label{eq:polygonalinterpolation} \tag{A1}
\mathsf{T}_j\subset \mathsf{T}_{j+1}\subset \Sigma \subset \mathsf{H}_{j+1}\subset \mathsf{H}_j    \qquad \qquad \forall j\geq 0
\end{equation}
and\footnote{Here and elsewhere below we simplify the notation as follows, $\operatorname{dist}(x,A)=\inf\{|x-a|:a\in A\}$ for a point $x$ and a set $A$, but $\operatorname{dist}(A,B)$ will be the Hausdorff distance between two sets $A$ and $B$.}
\begin{equation} \label{eq:generalhyponlevels}  \tag{A2}
  \forall \varepsilon>0\ \exists k\in \mathbb{N}:\quad
\{z\in \H_j: \operatorname{dist}\left(z,\partial \H_j\right) \geq \varepsilon \}\subset
\T_j\qquad  j\geq k.    
\end{equation}
This implies that
\[
 \Sigma=\bigcup_{j=0}^\infty \mathsf{T}_j=\operatorname{int}\left(\bigcap_{j=0}^\infty \mathsf{H}_j\right).
\]

Our main interest is for the boundary, $\partial\Sigma$, to be a fractal curve. The strategy that we present next combines the use of conformal mappings \cite{2007Banjai,2003Banjai} and a grad-div system formulation of the problem, with the quadratic projection method \cite{2004Levitin,2000Shargorodsky,1998Davies}.

\begin{workflow}[Upper and lower bounds for $k_\mathrm{g}$] \ 
\begin{itemize}
\item[\rm{W1)}]  \emph{Embedding of the region and domain monotonicity.} Find polygons  satisfying \eqref{eq:polygonalinterpolation}.
 By domain monotonicity, upper (and lower) bounds for the Poincar{\'e} constant in $\mathsf{H}_j$ (and $\mathsf{T}_j$) give upper (and lower) bounds for $\pconst$.
Below $\Omega_j$ denotes either $\mathsf{H}_j$ or $\mathsf{T}_j$. 
\item[\rm{W2)}]   \emph{Conformal transplantation.} Determine conformal maps $\Omega_{0}\longrightarrow \Omega_j$. The eigenvalue problem on $\Omega_j$ is transformed into a pencil eigenvalue problem on $\Omega_0$ with a singular $j$-dependent right hand side. 
\item[\rm{W3)}]  \emph{Formulation as a system.} For fixed $j$, write the pencil eigenvalue problem on $\Omega_0$ as a first order system involving the gradient operator, the divergence operator and singular coefficients. 
\item[\rm{W4)}]  \emph{Computation of the upper and lower bounds.} Compute enclosures for the smallest positive eigenvalues of the singular first order systems by means of a pollution-free projection method. To make this concrete we choose the quadratic projection method. 
\end{itemize}
\end{workflow} 

Below we often refer to the fixed $j$ in any of the blocks W1)-W4), by saying that the relevant datum is associated to the level $j$.

In this scheme, our precise hypotheses on $\Sigma$ and $\partial \Sigma$ are as  follows.

\begin{assumptiona}
The region $\Sigma\subset \mathbb{R}^2$  is open and simply connected. There exist two sequences of simply connected polygons satisfying \eqref{eq:polygonalinterpolation} and \eqref{eq:generalhyponlevels}. Additionally, 
\begin{itemize}
\item[\rm{(A3)}] the boundary is given by
\[
    \partial \Sigma=\bigcup_{n=1}^N F_n
\]
where $F_n$ are self-similar curves each associated to an iterated function scheme, 
\item[\rm{(A4)}] the vertices of $\partial\T_j$ and $\partial \H_j$ are computable from these iterated function schemes.
\end{itemize} 
\end{assumptiona}

 As written in the Assumption~A, (A4) is not mathematically concrete and its proper formulation for particular cases is clarified in \S\ref{main1}.  If $\Sigma$ is a Koch snowflake, for example, $\T_j$ can be chosen to be the classical $j$th step of the construction starting from $\T_0$ an equilateral triangle. And $\mathsf{H}_j$  the less standard but well known  $j$th step of the construction starting from a hexagon $\H_0$, see \cite[Plate~37]{1977Mand}, \cite[Plate~43]{1982Mand} and Figure~\ref{fig:poly_levels} below. Hence, in this case (A4) is guaranteed by construction.

In \S\ref{main1} we describe and justify the block W1) in the workflow given this hypothesis. In particular the Assumption~A covers classical domains with fractal boundary, but in the construction of $\T_j$ and $\H_j$, both \eqref{eq:polygonalinterpolation} and \eqref{eq:generalhyponlevels} demand careful attention. We determine this specific construction for: a Koch snowflake, a C{\`e}saro (generalised Koch) snowflake of any angle \cite{1909Cesaro}, a quadric island \cite[Plate~49]{1977Mand}
and a Gosper-Peano island \cite[Plate~47]{1977Mand}, \cite{1976Gardner}.
In lemmas~\ref{convergencegeneric}, \ref{lem:def_sf} and \ref{generalrate}, we establish convergence of the eigenvalue bounds of block W1) as $j\to\infty$, by applying directly an estimate of Pang \cite{1997Pang}. 

The singularities of the derivative of the conformal maps associated to block W2) determine the domain of the grad-div matrix operator in the formulation of block W3). In \S\ref{main2} we conduct a detailed analysis of these singularities. In \S\ref{main3}, on the other hand, we describe the precise operator-theoretical setting of the grad-div singular eigenvalue problem. Several technical details in this respect  are given in the Appendix~\ref{specmat}.

The reduction to a grad-div formulation proposed in block W3) is not standard in the context of eigenvalue computation for the Laplacian. It is also counterintuitive, as one is left with an indefinite eigenvalue problem which is prone to spectral pollution due to variational collapse. However, an order reduction of the differential operator often improves the accuracy of a non-pollution projection method \cite{2017BBB,2016Boulton, 2014BBB}. We have chosen the quadratic method. In \S\ref{main4} we survey this method and establish the details of its justification.

  \S\ref{calculations} is devoted to a full concrete implementation of W1)-W4) for $\Sigma$ a Koch snowflake inscribed in the unit circle. We report on details of our calculations leading towards the following estimate
\begin{equation} \label{finalbounds}
   {13.11601}   \leq \frac{1}{\pconst} \leq  {13.11623}.
 \end{equation}
See Table~\ref{table1}. Remarkably, by formulation of the method, neither the computation nor the validation of these bounds relies on asymptotic arguments.

\section{Embedding of the region and domain monotonicity \label{main1}}
Consider the eigenvalue problem associated with the Dirichlet Laplacian,
\begin{equation} \label{DirichletLaplacian}
  \begin{aligned}    
    -\Delta u&=\omega^2 u & &\text{in } \Omega \\ u&=0 & &\text{on }\partial \Omega
      \end{aligned}
\end{equation}
on a simply connected open set $\Omega$. By virtue of the classical min-max principle, $\pconst=\frac{1}{\omega_1^2}$ where
$\omega_1\equiv \omega_1(\Sigma)>0$ is the square root of the first eigenvalue of
\eqref{DirichletLaplacian} for $\Omega=\Sigma$. Domain monotonicity ensures that 
\begin{equation}   \label{domainmonotonicity}
\Omega\subset \tilde{\Omega} \quad \Rightarrow \quad 
 \omega_{k}^2(\tilde{\Omega}) \leq \omega_k^2(\Omega).
\end{equation}
We will, in particular, repeatedly use this property for $k=1$ and $k=2$. 

Combining the embedding condition \eqref{eq:polygonalinterpolation} with \eqref{domainmonotonicity}  yields
\[
        \omega_1^2(\mathsf{H}_{j})\leq   \omega_1^2(\Sigma)  \leq\omega_1^2(\mathsf{T}_{j}).
\]
 The next lemma is crucial to our analysis. Its proof is an immediate consequence of a uniform estimate on $\omega_j^2(\Omega)$ from inner approximations of $\Omega$ established by Pang \cite[Theorem~1.1]{1997Pang}. Here and elsewhere $\mathrm{j}_{0,1}\approx 5.784$ is the first eigenvalue of \eqref{DirichletLaplacian} for $\Omega$ the unit disk.

\begin{lemma}\label{convergencegeneric} 
Let $\{\T_j\}_{j=0}^\infty$ and $\{H_j\}_{j=0}^\infty$ be two families of open simply connected polygons, such that $\T_j\subset\T_{j+1}\subset\H_{j+1}\subset\H_j$ and such that \eqref{eq:generalhyponlevels} holds true.  Let
\[
    C=\frac{2^9\mathrm{j}^4_{0,1}S^{9/4}}{3\pi^{9/4}R^7}  
\]
where   $S=|\H_0|$ and $R$ is the  inradius of $\T_0$,
\[
     R=\sup_{z\in \T_0} \dist(z,\partial \T_0).
\]
For $\varepsilon>0$, the same $k\in\mathbb{N}$ satisfying \eqref{eq:generalhyponlevels} yields 
\[
    \omega_1^2(\mathsf{T}_j)-\omega_1^2(\mathsf{H}_j)\leq C \varepsilon^{1/2} \qquad \forall j\geq k.
\]
\end{lemma}
\begin{proof}
Apply directly \cite[Theorem~1.1]{1997Pang}, observing that $|\H_0|\geq |\H_j|$ and that the  inradius of $\T_0$ is less than or equal the inradiuses of the $\H_j$. 
\end{proof}

Thus \eqref{eq:polygonalinterpolation} and \eqref{eq:generalhyponlevels} imply that, in the context of the block W1), 
\begin{equation}  \label{convergenceoflevels}
\omega_1^2(\mathsf{T}_j)\downarrow \omega_1^2(\Sigma) \quad \text{and} \quad
\omega_1^2(\mathsf{H}_j)\uparrow \omega_1^2(\Sigma), \qquad j\to \infty.
\end{equation}
Note that weaker versions of this result can also be established if $\T_j$ or $\H_j$ are not nested.

\subsection{The classical Koch snowflake} \label{classsnofla}
We first describe a specific construction of $\T_j$ and $\H_j$ for a Koch snowflake. Scale $\Sigma$ such that it is inscribed in a unit circle. Let $\mathsf{T}_0$ be an equilateral triangle of side length $\sqrt{3}$ and  $\mathsf{H}_0$  a hexagon of side length 1, both  centred at 0, such that $\mathsf{T}_0$ is inscribed in $\mathsf{H}_0$. The polygon $\mathsf{T}_j$ is constructed by attaching to the central third of each side of $\mathsf{T}_{j-1}$ an equilateral triangle, whereas the polygon $\mathsf{H}_j$ is constructed from $\mathsf{H}_{j-1}$ by subtracting an equilateral triangle. The two procedures have the following generators, with $\mathsf{T}_j$ at the top and $\mathsf{H}_j$ at the bottom.
 \begin{center}
   \includegraphics[width=.5\textwidth]{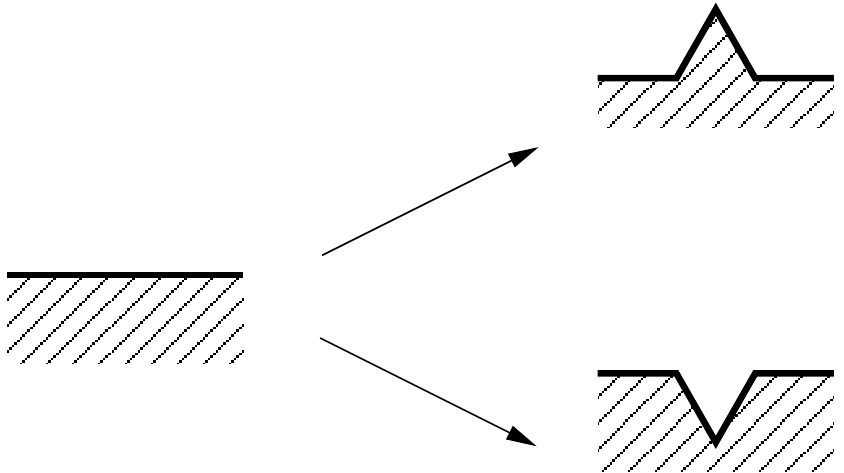}
   \end{center}
The resulting polygons for levels $j = 0,1,2$ are shown in Figure~\ref{fig:poly_levels}.

 \begin{figure}
   \centering
   \includegraphics[width=.75\textwidth]{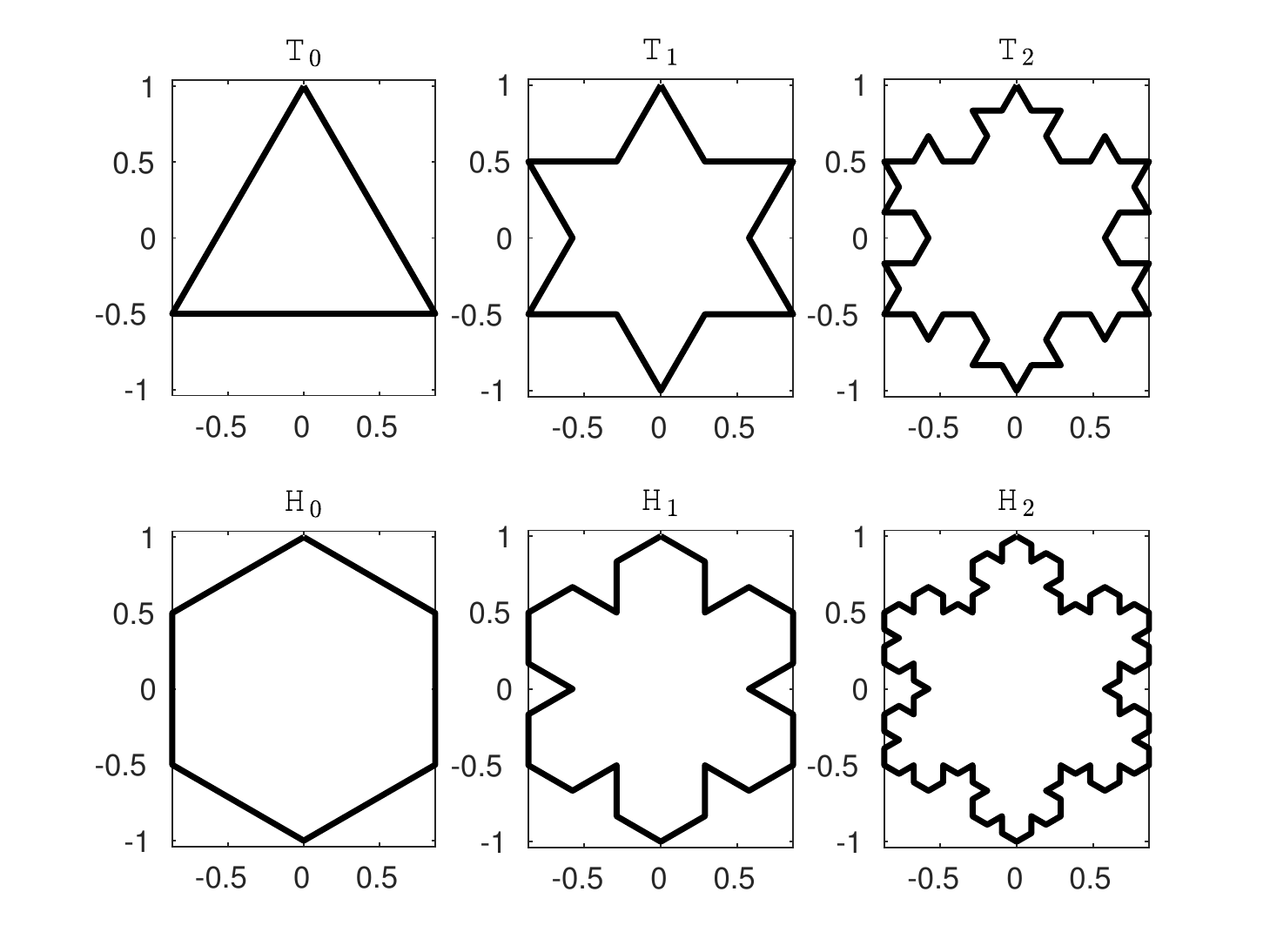}
   \caption{Three levels of $\mathsf{T}_j$ and $\mathsf{H}_j$ for the Koch snowflake.}
   \label{fig:poly_levels}
 \end{figure}

\begin{lemma} \label{lem:def_sf}
Let $\T_j$ and $\H_j$ be as described in the preceding paragraph. Then, $\T_j\subset \T_{j+1}\subset \H_{j+1}\subset \H_j$ for all $j\geq 0$ and \eqref{eq:generalhyponlevels} holds true.
Moreover
\begin{equation} \label{eq:convergencerate}
    \omega_1^2(\mathsf{T}_j)-\omega_1^2(\mathsf{H}_j)\leq \frac{\mathrm{j}_{0,1}^4 3^{3/4}}{2^{5/4}\pi^{9/4}} \left( \frac{1}{\sqrt{3}} \right)^{j} \qquad 
\forall j \in \mathbb{N}.
\end{equation}
\end{lemma}
\begin{proof}
We  claim that polygon $\H_j$ is obtained by attaching to each edge of $\T_j$ an isosceles triangle whose base is this edge and whose height is $\frac1{2\sqrt{3}}$ times the length of the edge.  From this, it follows that $\T_j \subset \H_j$ and that
\[
 \left\{ z\in \mathsf{H}_j\,:\,\operatorname{dist}(z,\partial \mathsf{H}_j)\geq \frac{1}{3^{j+1}}\right\}\subset \mathsf{T}_j.
\]
The latter implies \eqref{eq:generalhyponlevels}. Then, \eqref{eq:convergencerate} follows from Lemma~\ref{convergencegeneric}, taking $\varepsilon=\frac{1}{3^{j+1}}$ and $R=1$.

It remains to prove the claim by induction. It is not difficult to check that the claim holds for $j = 0$. Next, assume that it holds for some $j = k > 0$. After rotation and translation let  $AB$ with $A = (0,0)$, $B = (\ell,0)$ and $\ell = 3^{-k}\sqrt{3}$ be an edge of $\T_k$; see Figure~\ref{fig:lemma_proof}.
Then by assumption $BC$ and $CA$ are edges of $\H_k$ where $C = (\tfrac{\ell}2,\tfrac{\ell}{2\sqrt{3}})$. By the definition of polygons $\T_j$, $AD$ and $DC$ are edges of $\T_{k+1}$ where $D = (\tfrac{\ell}{3},0)$. Further, by the definition of polygons $\H_j$, $AE$, $ED$, $DF$ and $FC$ are edges of $\H_{k+1}$ where $E = (\tfrac{\ell}{6},\tfrac{\ell}{6\sqrt{3}})$ and $F = (\tfrac{\ell}{3},\tfrac{\ell}{3\sqrt{3}})$. As triangles $ADE$ and $DCF$ are of the required shape, we have proved the claim for $j = k+1$ and hence by induction for all $j \geq 0$.
\end{proof}

\begin{figure}
  \centering
\includegraphics[width=.5\textwidth]{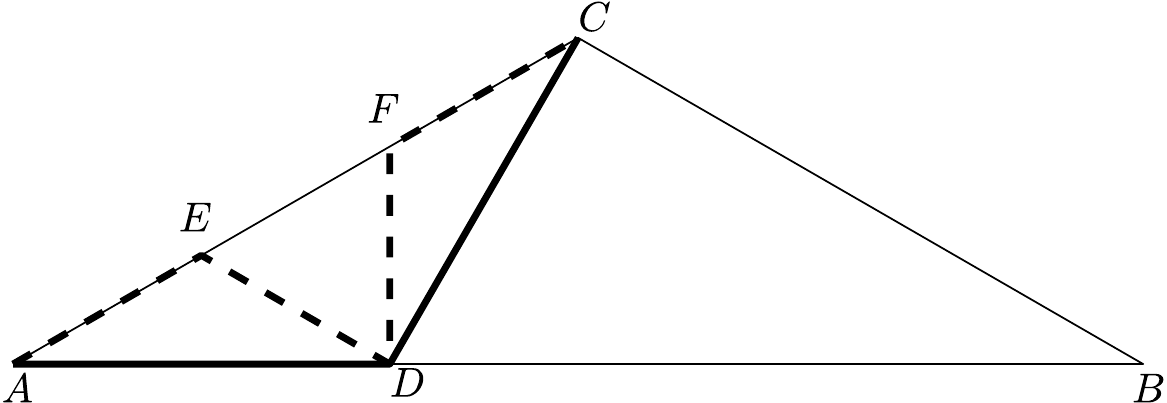}  
  \caption{A diagram in aid of Lemma~\ref{lem:def_sf} explaining the connection between $\T_j$ and $\H_j$.}
  \label{fig:lemma_proof}
\end{figure}

Let 
\[
    \Omega_{\mathsf{T}}=\bigcup_{j=0}^\infty \mathsf{T}_j \qquad \text{and} \qquad
   \Omega_{\mathsf{H}}=\operatorname{int}\left(\bigcap_{j=0}^\infty \mathsf{H}_j\right).
\]
As a consequence of the previous lemma $\Omega_{\mathsf{T}}=\Omega_{\mathsf{H}}$. When referring to the Koch snowflake below, without further mention, we are setting $\Sigma=\Omega_{\mathsf{T}}= \Omega_{\mathsf{H}}$.
In this case, (A1) and (A2) hold true, the condition (A3) is given by 3 components $F_n$, which are Koch curves, and (A4) is intrinsic to the construction.     

\begin{remark} \label{conjecturerate}
In the proof of Lemma~\ref{convergencegeneric} we have employed the general estimate of Pang, which applies to a large class of regions.  As it is natural to expect and as we shall see from numerical evidence presented in \S\ref{calculations}, the rate factor $\frac{1}{\sqrt{3}}$ in \eqref{eq:convergencerate} is sub-optimal for the specific approximation of the Koch snowflake by $\T_j$ and $\H_j$. Our evidence suggests that 
\[
     \omega_1^2(\mathsf{T}_j)-\omega_1^2(\mathsf{H}_j)\approx C \varrho^{j} \quad \text{where} \quad 
     \varrho \approx 0.35958, \; C \approx 5.8688.
\]
See Table~\ref{table2}.
\end{remark}

\subsection{General Koch curves}

The sequences $\{\T_j\}$ and $\{\H_j\}$ can be constructed in similar way as above, if $\partial{\Sigma}$ is the union of several boundary components $F_n$ in (A3), all of them being Koch curves with a generator of the form 
\begin{center}
   \includegraphics[width=.6\textwidth]{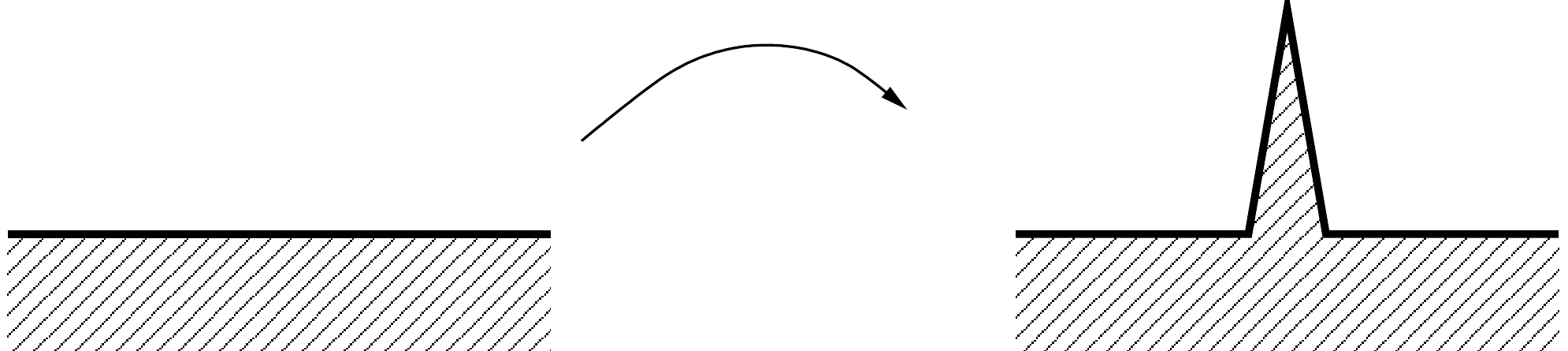}
   \end{center}
In the L-system language \cite{1980RozenbergSalomaa}, $F_n$ has generator \verb:"F"->"F+F--F+F": with the same angle $0<\alpha<\pi/2$. In particular, let $\T_0$ be a regular $N$-gon. The polygon $\mathsf{T}_j$ is obtained by attaching to the center of each side of $\mathsf{T}_{j-1}$ an isosceles triangle whose sides are equal to the length of the other two sides remaining in the segment.  Polygons $\H_j$ are then constructed from these polygons $\T_j$ in similar way as in the proof of Lemma~\ref{lem:def_sf}. Namely, $\H_j$ is obtained by attaching to each edge of $\T_j$ an isosceles triangle whose base is this edge and whose height is $\left(\tan\frac{\alpha}{2}\right)$ times the length of the edge.

A similar approach also works for antisnowflakes. For example $F_n$ could be C{\`e}saro curves \cite{1909Cesaro} of the form \verb:"F"->"F-F++F-F": of any angle swapping the way the polygons $\T_j$ and $\H_j$ are constructed.

In both these constructions, the validity of the Assumption~A is ensured by a statement very similar to Lemma~\ref{lem:def_sf} with a proof which is almost identical. We omit further details.

\subsection{Other fractal curves with simple generators}
We now establish a principle for constructing inner-outer polygonal approximations satisfying the Assumption~A, which apply to other classes of regions with piecewise self-similar boundary. As we shall see later, this includes classical fractals such as those shown in Figure~\ref{fig:fractals}. We begin by developing a necessary notation and formulating a general result about simply connected polygons.

\begin{figure}
  \centering
  \includegraphics[width=.95\textwidth]{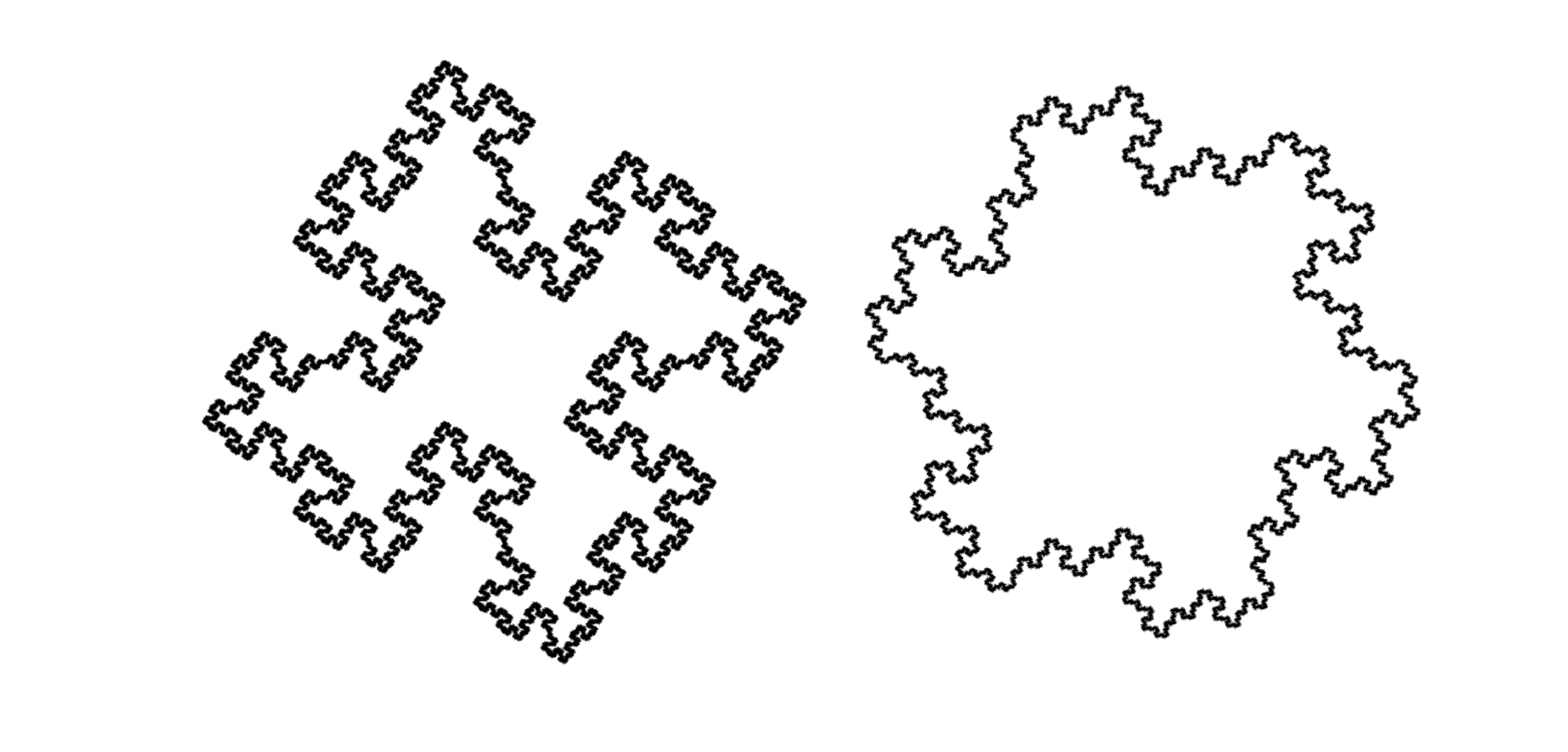}
  \caption{Quadric island (left) and Gosper-Peano island (right).}
  \label{fig:fractals}
\end{figure}

The open $\varepsilon$-neighbourhood in Euclidean distance of a set $\mathsf{S}\subset \mathbb{R}^2$ will be written as $(\mathsf{S})_{\varepsilon}$. Let $\mathsf{S}\subset \mathbb{R}^2$ be an open bounded polygon such that $\partial \mathsf{S}$ is a closed Jordan curve. Order the segments on the boundary of $\mathsf{S}$, $\underline{\sigma}_k$, contiguously: such that
\[
    \partial \mathsf{S}=\bigcup_{k=1}^M \underline{\sigma}_k  
\]
and $\underline{\sigma}_j\cup \underline{\sigma}_{j+1}$ is a continuous curve for all $j=1:M$ denoting $\underline{\sigma}_{M+1}=\underline{\sigma}_{1}$. 

Let $\mathbb{V}(\mathsf{S})$ be the set of vertices/corners on $\partial \mathsf{S}$. Let $\varepsilon>0$. For a corner $A\in \mathbb{V}(\mathsf{S})$ of inner angle $\beta\equiv \beta(A)$, denote by $A^\mathrm{i}_\varepsilon$ the unique point on the segment bisecting $\beta(A)$ at distance $\frac{\varepsilon}{\sin \frac{\beta}{2}}$ from $A$ and by $A^\mathrm{o}_\varepsilon$ the unique point on the segment bisecting the complementary angle to $\beta(A)$ at distance $\frac{\varepsilon}{\sin \frac{\beta}{2}}$  from $A$. That is $A^\mathrm{i}_\varepsilon$ and $A^\mathrm{o}_\varepsilon$ are the endpoints of a segment that also passes through $A$. Because $\partial \mathsf{S}$ is a closed Jordan curve which is piecewise linear, and all angles $\beta(A)\in (0,2\pi)$, there exists $\varepsilon_0$ such that  $A^\mathrm{i}_\varepsilon\in \mathsf{S}$ and $A^\mathrm{o}_\varepsilon\not\in \mathsf{S}$ for all $0<\varepsilon<\varepsilon_0$ and $A\in \mathbb{V}(\mathsf{S})$. Indeed this is a consequence of the fact that the segments $[A^\mathrm{i}_\varepsilon,A^\mathrm{o}_\varepsilon]$ have midpoint $A$ and $A^\mathrm{i,o}_\varepsilon\to A$ as $\varepsilon \to 0$.

For the side $\underline{\sigma}_j=[A,B]$, set $\mathsf{Q}^j_{\varepsilon} \equiv \mathsf{Q}^j_{\varepsilon}(\mathsf{S})$ to be the open quadrilateral (trapezoid or rhomboid) with vertices $A^\mathrm{i}_\varepsilon,\, A^\mathrm{o}_\varepsilon,\, B^\mathrm{i}_\varepsilon,\, B^\mathrm{o}_\varepsilon$. That is two sides of $\mathsf{Q}_{\varepsilon}^j$ are parallel to $\underline{\sigma}_j$ at distance $\varepsilon$ from this segment, another is $[A^\mathrm{i}_\varepsilon, A^\mathrm{o}_\varepsilon]$ and the other is $[B^\mathrm{i}_\varepsilon, B^\mathrm{o}_\varepsilon]$. 
For $0 < \varepsilon < \varepsilon_0$, we will write
\[
[\mathsf{S}]^{\mathrm{i}}_\varepsilon=\mathsf{S}\setminus \left(\bigcup_{j=1}^M \overline{\mathsf{Q}_{\varepsilon}^j}\right)
\quad \text{and} \quad
[\mathsf{S}]^{\mathrm{o}}_\varepsilon=\mathsf{S}\cup \mathrm{int}\left(\bigcup_{j=1}^M \overline{\mathsf{Q}_{\varepsilon}^j}\right).
\] 
Then 
\begin{equation*}  \label{eq:forinterp1}
    [\mathsf{S}]^{\mathrm{i}}_\varepsilon\subset \mathsf{S} \subset [\mathsf{S}]^{\mathrm{o}}_\varepsilon
\qquad \text{and} \qquad (\partial \mathsf{S})_{\varepsilon}\subset \bigcup_{k= 1}^M \overline{\mathsf{Q}_\varepsilon^k} = 
\overline{[\mathsf{S}]^{\mathrm{o}}_\varepsilon}\setminus
       {[\mathsf{S}]^{\mathrm{i}}_\varepsilon}.
\end{equation*}
Let
\[
    \beta_{\mathrm{m}}(\mathsf{S})=\max_{A\in \mathbb{V}(\mathsf{S})} \left\{ \beta(A),2\pi-\beta(A)\right\}\in (\pi,2\pi)
\]
be the maximal angle among all angles (inner and outer) in $\partial \mathsf{S}$. Since
\[
      \dist \left(y, \partial  \bigcup_{k= 1}^M \overline{\mathsf{Q}_\varepsilon^k}\right)\leq
      \frac{\varepsilon}{\sin \frac{\beta_{\mathrm{m}}(\mathsf{S})}{2}} \qquad \forall y\in \partial\mathsf{S},
\]
then 
\begin{equation} \label{eq:forinterp3}
     [\mathsf{S}]^{\mathrm{o}}_\varepsilon \setminus \overline{[\mathsf{S}]^{\mathrm{i}}_\varepsilon}
     \subset (\partial \mathsf{S})_{\tilde{\varepsilon}}  
\quad \text{for}\quad   \tilde{\varepsilon}=\frac{\varepsilon}{\sin \frac{\beta_{\mathrm{m}}(\mathsf{S})}{2}}.
\end{equation}

The proof of the following lemma is a straightforward consequence of the fact that $\dist(\overline{[\mathsf{S}]^{\mathrm{o}}_\varepsilon \setminus [\mathsf{S}]^{\mathrm{i}}_\varepsilon},\partial \mathsf{S})\to 0$ as $\varepsilon \to 0$.

\begin{lemma} \label{thetwoneighbourhoods}  
There exists $\varepsilon_1>0$ such that the following holds true for all $0<\varepsilon<\varepsilon_1$. The intersection $\mathsf{Q}_{\varepsilon}^m \cap \mathsf{Q}_{\varepsilon}^n = \varnothing$ for $m  \neq n$ and the sets 
$[\mathsf{S}]^{\mathrm{i}}_\varepsilon$ and $[\mathsf{S}]^{\mathrm{o}}_\varepsilon$ are simply connected. 
\end{lemma}

We now determine concrete sufficient conditions for a region $\Sigma$ to satisfy the Assumption~A.  

\begin{hypothesisG} 
Let $\Sigma$ be an open planar region such that $\partial\Sigma=\bigcup_{n=1}^N F_n$ is a Jordan curve and $F_n$ are self-similar curves. Let $\Sigma_j$ be the resulting open pre-fractal polygons. That is, $\partial\Sigma_j$ is the union of the steps $j$ in the generation of the $F_n$ for $n=1:N$. Suppose that all $\partial \Sigma_j$ are Jordan curves, that there is $\beta_0\in(\pi,2\pi)$ such that $\beta_{\mathrm{m}}(\Sigma_j)\leq \beta_0$ and that all the polygons $\Sigma_j$ have all their sides of equal length, $\ell_j$, where $\ell_j>\ell_{j+1}\to 0$.  Additionally write $\mathsf{Q}_{\varepsilon}^{m,j} = \mathsf{Q}_{\varepsilon}^{m}(\Sigma_j)$. Suppose that there exists a constant $\delta>0$ such that for all $j\in\mathbb{N}$
\begin{itemize}
\item[{\rm{(G1)}}]  $\mathsf{Q}_{\delta \ell_j}^{m,j} \cap \mathsf{Q}_{\delta \ell_j}^{n,j}= \varnothing$ for $m  \neq n$, 
\item[{\rm(G2)}] the sets
\[
    \T_j=[\Sigma_{j}]^{\mathrm{i}}_{\delta \ell_{j}} \qquad \text{and} \qquad
    \H_j=[\Sigma_{j}]^{\mathrm{o}}_{\delta \ell_{j}}
\]
are simply connected,
\item[{\rm(G3)}]  
 $\left(\H_{j+1} \setminus \T_{j+1}\right) \subset \left(\H_{j} \setminus \T_{j}\right)$ and
\item[{\rm(G4)}]  $\bigcap\limits_j \T_j \neq \varnothing$.
\end{itemize}
\end{hypothesisG}

\begin{lemma}  \label{lem:generalconst}
If the planar region $\Sigma$ satisfies the Hypothesis~G, then it satisfies the Assumption~A.
\end{lemma}
\begin{proof}
Hypothesis G already contains (A3).  The condition (G1) implies that the vertices of $\T_j$ and $\H_j$ are respectively the vertices of 
$
     \mathsf{Q}^{m,j}_{\delta \ell_j}
$  
and we have an explicit formula for them in terms of the vertices of $\Sigma_j$. This is a concrete realisation of  (A4).

Now, consider (A1). By construction we have that $\T_j \subset \Sigma_j \subset \H_j$. From (G2) it follows that $\mathbb{R}^2$ is split into three disjoint connected regions $\T_{j+1}$, $\H_{j+1}\setminus \T_{j+1}$ and $\H_{j+1}^c = \mathbb{R}^2 \setminus \H_{j+1}$.  From (G3) we have that  $\H_j \cap (\H_{j+1}\setminus \T_{j+1}) = \varnothing$. Also $\H_{j+1}^c \not\subset \T_{j+1}$ because the left-hand side is unbounded. Hence, $\H_j^c \subset \H_{j+1}^c$ and so  \begin{equation} \label{oneside} \H_{j+1} \subset \H_j.\end{equation}
On the other hand (G3) also implies $\T_j \cap (H_{j+1}\setminus \T_{j+1})= \varnothing$, hence either $\T_j \subset \T_{j+1}$ or $\T_j \subset \H_{j+1}^c$. The latter contradicts (G4) and thus \begin{equation} \label{otherside} \T_j \subset \T_{j+1}.\end{equation} 
From \eqref{oneside} and \eqref{otherside} follows that 
\begin{equation}   \label{eq:forA1andA2}
     \T_j\subset \Sigma_{j+k} \subset \H_j   \qquad \qquad \forall k\in \mathbb{N}.
\end{equation}
Since the Hausdorff limit of a family of compact sets contains the intersection of this family and it is contained in the union of this family, (A1) is a consequence of \eqref{eq:forA1andA2}.

We finally show (A2). Let $\varepsilon>0$ and let $k>0$ be such that 
\[
 \ell_k<\frac{\varepsilon \sin \frac{\beta_0}{2}}{2\delta}.
\] 
For $j>k$, from the fact that $\ell_j$ are decreasing and in the notation of \eqref{eq:forinterp3},
\[
   \widetilde{\delta \ell_j}=\frac{\delta\ell_j}{\sin \frac{\beta_{\mathrm{m}} (\Sigma_j)}{2}}< \frac{\delta \ell_k}{\sin \frac{\beta_0}{2}} <\frac{\varepsilon}{2}.
\]
Then, from \eqref{eq:forinterp3} it follows that 
\[
   \dist\left(x,\partial [(\partial \Sigma_j)_{\widetilde{\delta \ell_j}}]\cap \Sigma_j\right)<2\widetilde{\delta\ell_j}<\varepsilon \qquad \forall x\in \partial \H_j
\] 
and
\[\partial[(\partial\Sigma_j)_{\widetilde{\delta \ell_j}}]\cap \T_j=
\partial[(\partial\Sigma_j)_{\widetilde{\delta \ell_j}}]\cap \Sigma_j=
\partial[(\partial\Sigma_j)_{\widetilde{\delta \ell_j}}]\cap \H_j.
\]
Hence, $\H_j\setminus (\partial \H_j)_\varepsilon \subset \T_j$. Thus (A2) holds true.
\end{proof}

In the Hypothesis~G, the conditions (G1) and (G2) contrast with (G3). That is, by Lemma~\ref{thetwoneighbourhoods}, (G1) and (G2) are satisfied for $\delta$ small enough. They do not depend on the relation between different levels. But, as we shall see in the examples of the quadric and Gosper-Peano islands, (G3) will break for $\delta$ too close to 0. The condition (G4) is independent of any of this.

We now show a convergence estimate in the context of block W1) for the ground eigenvalues of $\T_j$ and $\H_j$ defined as in the Hypothesis~G.  

\begin{lemma} \label{generalrate}
Let the planar region $\Sigma$ satisfy the Hypothesis~G. Let $C>0$ be the constant of Lemma~\ref{convergencegeneric}. Then
\begin{equation} \label{eq:genregionconvergence}
       \omega_1^2(\mathsf{T}_j)-\omega_1^2(\mathsf{H}_j)\leq  \frac{C\sqrt{2\delta}}{\sin (\frac{\beta_0}{2})^{\frac12}} \ell_j^{\frac12}\downarrow 0 \qquad j\to \infty.
\end{equation}
\end{lemma}
\begin{proof}
Under the conditions of the Hypothesis~G,
\[
   \left\{z\in \H_j:\operatorname{dist}(z,\partial \H_j)\geq \frac{2\delta \ell_j}{\sin \frac{\beta}{2}}\right\}\subset \T_j.
\]
The conclusion follows by applying Lemma~\ref{convergencegeneric}
and recalling that $\ell_j\to 0$ monotonically.
\end{proof} 

In this lemma, it is remarkable that the convergence rate of the eigenvalue gap is directly related to the decay rate of the $\ell_j$. We now examine the Hypotheses~G for  two classical fractal region.

\subsection{Quadric islands}
Let $\Sigma$ be a quadric island \cite[Plate~49]{1977Mand} constructed as follows. Begin with $\Sigma_0$ a square of side $1$. Let $F_n$ for $n=1:4$ be constructed using the generator
 \begin{center}
   \includegraphics[width=.5\textwidth]{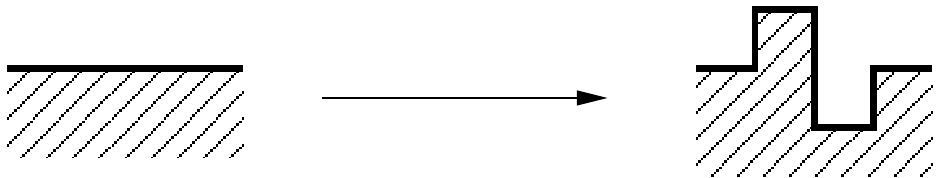}
   \end{center}
See Figure~\ref{fig:fractals} left. 
 
Suppose that the sides of $\Sigma_0$ are aligned with the horizontal and vertical axes. Then, $\partial \Sigma_j$ is the union of two families, each comprising segments of equal length. One of these families is made of segments aligned with the horizontal axis and the other is made of segments aligned with the vertical axis. It is readily seen that 
\[ \ell_j=\frac{\ell_{j-1}}{4}=\frac{1}{4^j} \qquad \text{and} \qquad \beta_{\mathrm{m}}(\Sigma_j)=\frac{3\pi}{2}.\] 

The trapezoids and rhomboids which jointly form $\mathsf{H}_j\setminus \overline{\mathsf{T}_j}$ always have two opposite edges parallel to the axes. Then,   (G1) holds true for all $\delta\leq \frac12$. Moreover, $\partial \mathsf{T}_j$ and $\partial \mathsf{H}_j$ are Jordan curves, so (G2) as well as (G4) are satisfied for all $\delta<\frac12$.

On the other hand, 
\[
    \operatorname{dist}(\partial \Sigma_j,\partial\Sigma_{j+1}) =\ell_{j+1}.
\]
Hence
\[
     \operatorname{dist}(\partial \Sigma_j,\partial[(\Sigma_{j+1})_{\delta \ell_{j+1}}]) =(1+\delta)\ell_{j+1}.
\]
Thus, the condition (G3) holds if and only if
\[
    \delta \ell_j\geq (1+\delta) \ell_{j+1}.
\] 
Clearing for $\delta$, we get that $\delta\geq \frac13$. See Figure~\ref{fig:quadric} (left).

This argumentation shows that for the quadric island with generator as above, the inner-outer interpolants $\mathsf{T}_j$ and $\mathsf{H}_j$, as in the Hypothesis~G, ensure the validity of the Assumption~A for all $\frac13\leq \delta<\frac12$.  And Lemma~\ref{generalrate} implies that, for block~W1), $\omega_1^2(\mathsf{T}_j)-\omega_1^2(\mathsf{H}_j)\to 0$ as $ j\to \infty$ at a rate at least $\frac{1}{2^{j}}$.

\begin{figure}
  \centering
  \includegraphics[width=.45\textwidth]{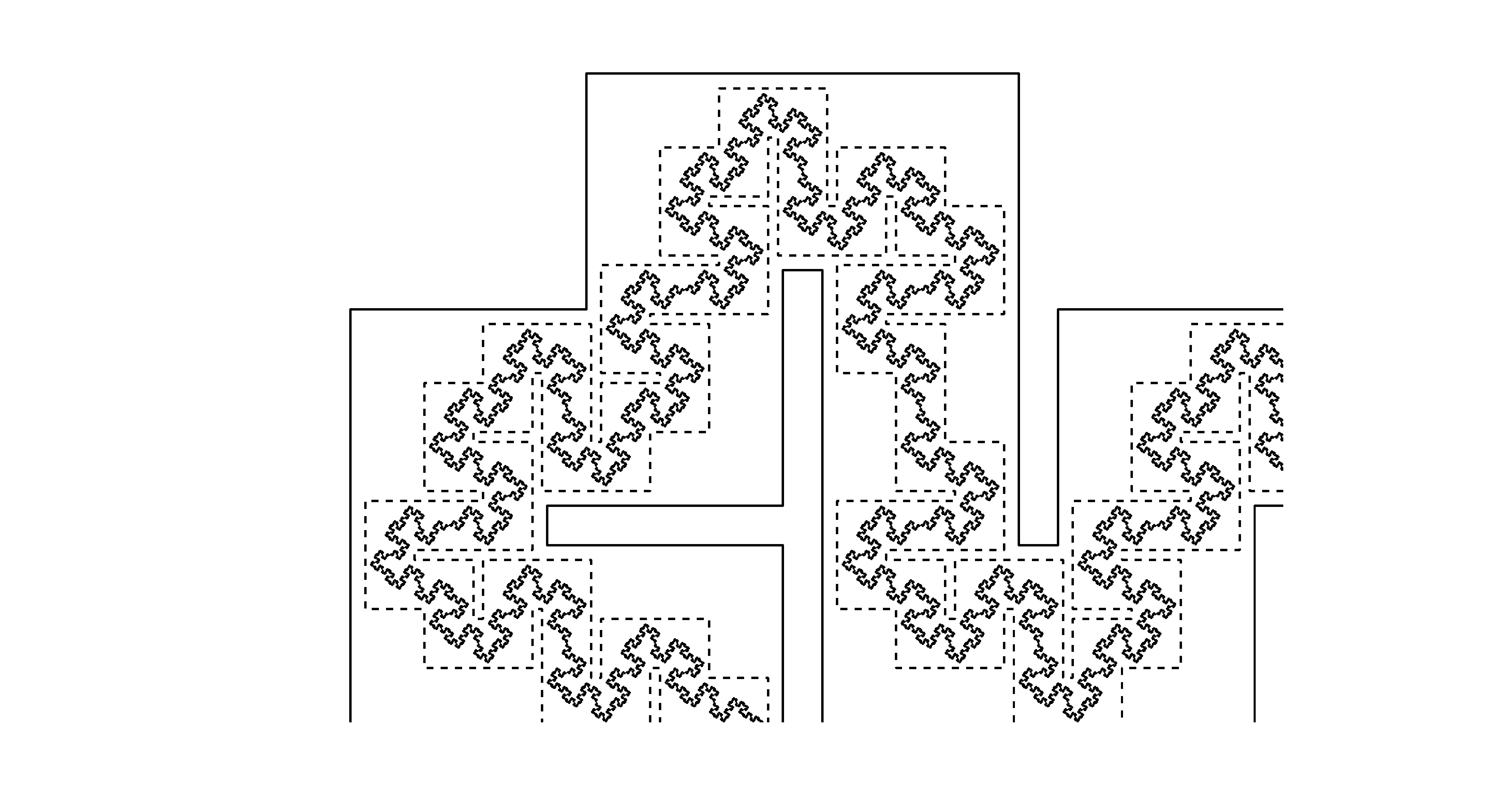}
  \includegraphics[width=.45\textwidth]{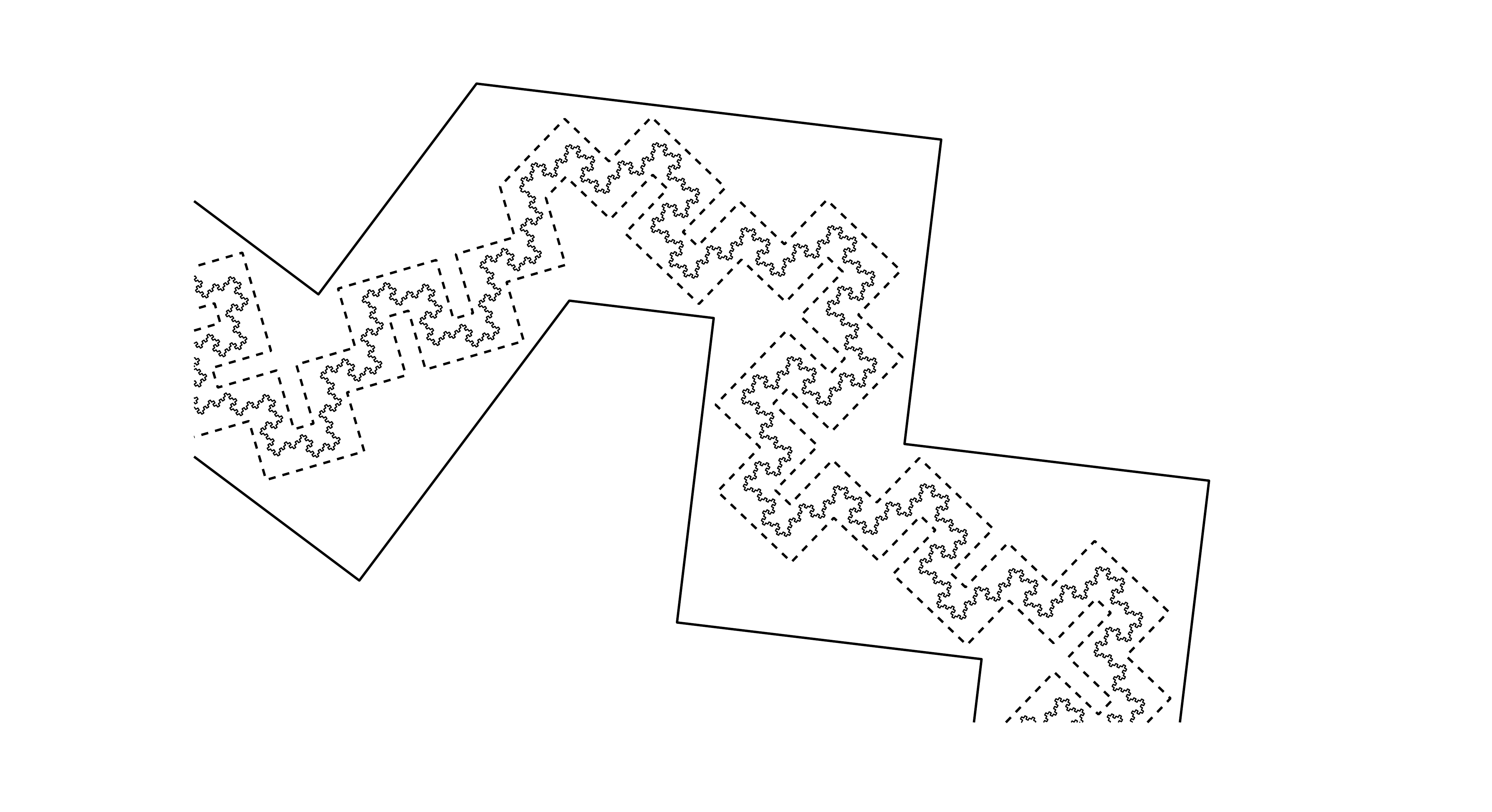}
  \caption{A section of the quadric island with $\T_j$ and $\H_j$, for $j=3,4$ and $\delta=0.4$ (left). A section of the Gosper-Peano island  with $\T_j$ and $\H_j$, for $j=4,6$ and  $\delta=0.48$ (right).}
  \label{fig:quadric}
\end{figure}

\subsection{Gosper-Peano islands}

As a final example, we consider $\Sigma$ to be a Gosper-Peano island \cite[Plate~47]{1977Mand}, in which $\Sigma_0$ is a hexagon of side $1$.
See Figure~\ref{fig:fractals} right. That is $F_n$ for $n=1:6$ is constructed using the generator
 \begin{center}
   \includegraphics[width=.5\textwidth]{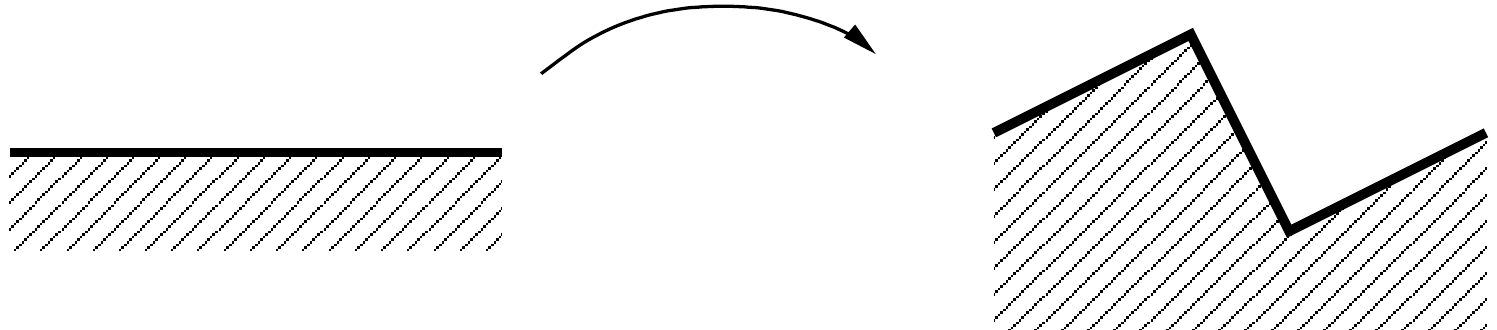}
   \end{center}

The boundary $\partial \Sigma_j$ is a union of equal length segments where \[\ell_{j}=\frac{\ell_{j-1}}{\sqrt{5}}=\frac{1}{5^{\frac{j}{2}}}.\] As we will see next, there is no $\delta>0$ ensuring all three conditions of the Hypothesis~G simultaneously. After that, we show how to overcome this difficulty, which is not caused by the fact that $\Sigma_0$ is a hexagon.

The inner angles of the vertices in each of the 6 components of $\partial \Sigma_j$ are either $\beta=\frac{\pi}{2}$ or $\beta = \frac{3\pi}{2}$. Each one of these components is an arrangement of segments (of the same length) aligned in two possible directions one perpendicular to the other. The inner angle between components is a fixed $\beta=\frac{2\pi}{3}$. So $\beta_{\mathrm{m}}(\Sigma_j)=\frac{3\pi}{2}$. Then, similar to the case of the quadric island, (G1), (G2) and (G4) are satisfied if and only if $0<\delta<\frac12$. 

Now,
\[
     \operatorname{dist}(\partial \Sigma_j,\partial[(\Sigma_{j+1})_{\delta \ell_{j+1}}]) =\frac{\ell_j}{5}+\frac{\delta \ell_{j+1}}{\frac{\sqrt{2}}{2}}\sin\left(\alpha+\frac{\pi}{4}\right)= \frac{\ell_j}{5}+\delta \ell_{j} \frac{3}{5}
\]
where $\alpha=\arcsin(1/\sqrt{5})$ is the angle between the $j$ and $j+1$ iteration. In order to satisfy (G3), we require 
\[
    \operatorname{dist}(\partial \Sigma_j,\partial[(\Sigma_{j+1})_{\delta \ell_{j+1}}])<\delta \ell_j.
\] 
Solving for $\delta$ gives $\delta> \frac12$. So indeed, there is no $\delta>0$ such that (G3) holds at the same time as (G1), (G2) and (G4).

To construct an approximating sequence of inner-outer polygons for $\Sigma$ in this case, we should therefore pick a subsequence of levels. For example, pick $\T_{2k}$ and $\H_{2k}$ for $k\in\mathbb{N}$.  Once again (G1), (G2) and (G4) will be satisfied for $0<\delta<\frac{1}{2}$, but now 
\[
    \operatorname{dist}(\partial \Sigma_{2k},\partial[(\Sigma_{2k+2})_{\delta \ell_{2k+2}}])\leq 
\frac{\ell_{2k}}{5}+\frac{\delta \ell_{2k+2}}{\frac{\sqrt{2}}{2}}=
\frac{\ell_{2k}}{5}+\frac{\delta \ell_{2k}\sqrt{2}}{5}.
\] 
And (G3) is guaranteed for
\[
    \operatorname{dist}(\partial \Sigma_{2k},\partial[(\Sigma_{2k+2})_{\delta \ell_{2k+2}}])<\delta \ell_{2k},
\] 
which implies $\delta>\frac1{5-\sqrt{2}}$ with the right hand side of this less than $1/2$. Hence all four conditions of the Hypothesis~G hold whenever we pick only the even levels for $\frac1{5-\sqrt{2}}\approx 0.2789<\delta<1/2$; see Figure~\ref{fig:quadric} (right). 

From Lemma~\ref{generalrate} we get  $\omega_1^2(\mathsf{T}_{2k})-\omega_1^2(\mathsf{H}_{2k})\to 0$ as $ k\to \infty$ at a rate at least $\frac{1}{5^{\frac{k}{2}}}$.

\section{Conformal transplantation \label{main2}}

We now turn to block W2).
Let $f\equiv f_j:\Omega_{0}\longrightarrow \Omega_j$ be the conformal map taking $\Omega_0$ to the corresponding level $\Omega_j$.
At times it will be useful to consider $\Omega_0$ as a subset of  the complex $z$-plane and $\Omega_j$ subset of the complex $w$-plane. When instead the polygons are viewed as lying in the two-dimensional real plane $\mathbb{R}^2$, the coordinates are denoted by $y=(y_1,y_2)\in \Omega_0$ and $x=(x_1,x_2)\in \Omega_j$. We also denote by $f:\mathbb{R}^2 \rightarrow \mathbb{R}^2$ the map given in real coordinates:
\[
    y\longmapsto x=f(y)=(f^1(y),f^2(y)).
\]

The following standard manipulations involving the composition map associated with $f$ will be useful in later sections.
As $f$ is analytic,  it satisfies the Cauchy-Riemann equations: \[\partial_1f^1=\partial_2f^2 \qquad \text{and}\qquad \partial_2f^1=-\partial_1f^2.\] Hence
\[
\nabla_y f =
\begin{bmatrix}
  \partial_1 f^1 & -\partial_2f^1\\
  \partial_2 f^1 & \partial_1 f^1
\end{bmatrix},
\]
\[
\det(\nabla_y f) = |f'|^2 \qquad \text{and} \qquad  (\nabla_y f)( \nabla_y f) ^T= 
|f'|^2
\begin{bmatrix}
  1 & 0\\ 0 & 1
\end{bmatrix},
\]
where
\[
     |f'|^2=(\partial_1f^1)^2+(\partial_2f^1)^2=(\partial_1f^2)^2+(\partial_2f^2)^2.
\]

Let $u \in C^2(\Omega_j)$  and $v = u \circ f$.
Then
\begin{equation*}
    \grad_y v=\nabla_y (u\circ f)=(\nabla_y f)^T \nabla_x u \circ f
\end{equation*}
and 
\begin{align*}
      \Delta_y v&=\div_y \grad_y v = \nabla_y\cdot \left[(\nabla_y f)^T \nabla_x u \circ f\right] \\
  &= [\nabla_x u\circ f] \cdot \underline{\Delta} f +  \operatorname{Tr} \left[ (\nabla_y f)^T (D_x^2 u\circ f) (\nabla_y f)  \right] \\ &=
0+ \operatorname{Tr} \left[ (\nabla_y f)^T (D_x^2 u\circ f) (\nabla_y f) \right] \\
&= |f'|^2 \Delta_x u\circ f.
\end{align*}
Here $D_x^2 u$ denotes the Hessian
\[
    D_x^2 u=\begin{bmatrix} \partial_{11}^2 u & \partial^2_{12} u \\
    \partial_{21}^2 u & \partial^2_{22} u\end{bmatrix}
\]
and $ \underline{\Delta} f$ the vector Laplacian 
\[
 \underline{\Delta} f=\begin{bmatrix} \Delta f^1 \\ \Delta f^2 \end{bmatrix}.
\]
Also
\begin{equation} \label{weightedsobolev}
    \int_{\Omega_{0}} |v(y)|^2 |f'(y)|^2 dy = \int_{\Omega_{j}} |u(x)|^2 dx<\infty
\end{equation}
and
\begin{equation} \label{honezero}
\int_{\Omega_{0}} |\nabla_{y} v(y) |^2 dy = \int_{\Omega_{j}} |\nabla_{x} u(x) |^2 dy<\infty,
\end{equation}
whenever $u\in H^1(\Omega_j)$.

The above calculations indicate that if $u$ is an eigenfunction of  \eqref{DirichletLaplacian} for $\Omega = \Omega_j$, then $v = u \circ f$ solves the transplanted eigenvalue problem
\begin{equation} \label{eq:transplanted}
  \begin{aligned}    
    -\Delta v&=\omega^2 |f'|^2 v & &\text{in } \Omega_0 \\ v&=0 & &\text{on }\partial \Omega_0. 
    \end{aligned}
\end{equation}
Moreover, if $v$ is an eigenfunction of \eqref{eq:transplanted}, then $u = v \circ f^{-1}$ is an eigenfunction of \eqref{DirichletLaplacian} associated with the same eigenvalue. As we shall see later in Theorem~\ref{operatorformulation} there is a one-to-one correspondence between the  eigenfunctions of the Dirichlet Laplacian on $\Omega_j$ and those of a selfadjoint operator associated to \eqref{eq:transplanted}. This is neither obvious nor an immediate consequence of classical principles, as $|f'|$ has zeros and poles on the boundary of the domain.

There are two reasons for preferring \eqref{eq:transplanted} over \eqref{DirichletLaplacian}. One is that, even though both $u$ and $v$ have singularities,  $v$ is more regular than $u$. The other reason is related to the fact that our polygons will have thousands of vertices. With techniques developed in \cite{2008Banjai,2003Banjai} we are able to efficiently and accurately compute the conformal map $f$, even in these extreme situations. Solving the eigenvalue problem on $\Omega_0$, especially for the ground eigenvalue, requires much simpler and smaller meshes than we would have needed on domains $\Omega_j$ with a highly complex boundary. This approach was also used in \cite{2007Banjai} to compute eigenvalues of fractal regions.

\subsection{The Schwarz-Christoffel maps for the Koch Snowflake} \label{sing_s-c}
In this sub-section we assume that $\Sigma$ is a Koch snowflake. Denote by $w_k$ the corners of the polygon $\Omega_j$ and by $z_k$ their pre-images under the map $f$, so that  $f(z_k) = w_k$. 
In the case of the polygons $\T_0$ and $\T_j$, we order the vertices so that $w_k = z_k$ for $k = 1,2,3$. That is, the first three vertices of $\mathsf{T}_j$ are the vertices of the original triangle $\mathsf{T}_0$. Similarly, for polygons $\H_0$ and $\H_j$ we require that $w_k = z_k$ for $k = 1,\dots, 6$.
The ordering of the remaining vertices is not important.

\begin{remark}
A conformal map between two domains is not unique, but can be made so by fixing three boundary points. For $\mathsf{T}_j$ this immediately ensures uniqueness. For $\mathsf{H}_j$, due to symmetries, we are able to  fix $6$ vertices. 
\end{remark}

We denote the interior angles of $\Omega_j$ by $\pi \alpha_k$. 
\begin{itemize}
\item If $\Omega_j = \T_j$, $\alpha_k = 1/3$ for $k = 1,2,3$ and $\alpha_k =  1/3$ or $4/3$ for $k > 3$. The total number of corners of $\T_j$ is 
\[
n(j)=4^{j}3=\underbrace{3+(4^j-1)}_{\alpha_k=1/3}+\underbrace{2(4^j-1)}_{\alpha_k=4/3}.
\]

\item If $\Omega_j = \H_j$, $\alpha_k = 2/3$ for $k = 1,\dots,6$ and $\alpha_k = 2/3$ or $5/3$ for $k > 6$. The total number of corners of $\H_j$ is 
\[
n(j)=4^{j}6=\underbrace{6+4(4^j-1)}_{\alpha_k=2/3}+\underbrace{2(4^j-1)}_{\alpha_k=5/3}.
\]
\end{itemize}

We construct the conformal map $f_j$ in two steps by means of an intermediate mapping onto $\mathsf{D}=\{|z|<1\}$. Set
\begin{equation}
  \label{eq:f_jdefn}
  f(z)\equiv f_j(z) = g_j \circ g_0^{-1}(z),
\end{equation}
where the conformal map $g_j:\mathsf{D}\longrightarrow \Omega_j$ is given by the Schwarz-Christoffel formula
\[
g_j(\xi) = A_j+C_j\int^\xi \prod_{k = 1}^{n(j)} (1-\zeta/\xi_k)^{\alpha_k-1}d\zeta.
\]
The position of  $\xi_k$, each  on the unit circle $\partial \mathsf{D}$, is initially unknown and needs to be computed by solving a non-linear system of equations \cite{2002LNT_SC}. Due to the symmetries, we can fix the first three/six pre-vertices, the pre-vertices that map to the three/six corners of the triangle/hexagon, all to be equally spaced points on $\partial \mathsf{D}$. In other words, $g_j(\xi_k) = w_k = z_k$, $k = 1,\dots,m$. Here and elsewhere $m = 3$ for $\T_j$ and $m = 6$ for $\H_j$. For all $k$ we have the following relationships
\[
z_k = g_0(\xi_k), \quad f_j(z_k) = g_j(\xi_k) = w_k.
\]

Note that
\begin{equation}
  \label{eq:f_jder}  
  \begin{split}    
f_j'(z) = \frac1{g_0'(g_0^{-1}(z))} g_j'(g_0^{-1}(z))
&=  \frac{C_j}{C_0} \prod_{k = m+1}^{n(j)} \left(1-g_0^{-1}(z)/\xi_k\right)^{\alpha_k-1}\\
&= \frac{C_j}{C_0} \prod_{k = m+1}^{n(j)} \left(1-g_0^{-1}(z)/g_{0}^{-1}(z_k)\right)^{\alpha_k-1}.
  \end{split}
\end{equation}
As expected, the singularities near the $m = 3$ or $m = 6$ fixed corners have disappeared. Integrating the above formula we obtain a Schwarz-Christoffel formula for $f$
\begin{equation}
  \label{eq:fj_SC}
f_j(z) = \tilde{A}_j+ \frac{C_j}{C_0} \int^z\prod_{k = m+1}^{n(j)} \left(1-g_0^{-1}(\zeta)/g_{0}^{-1}(z_k)\right)^{\alpha_k-1} d\zeta,
\end{equation}
where we left out the lower integration limit as it only influences the constant $\tilde{A}_j$. The following result on the regularity of $f$ is a consequence of this formula.

\begin{proposition}\label{prop:fj_smoothness}
Let  $f(z): \Omega_0 \rightarrow \Omega_j$  be the conformal map defined by \eqref{eq:fj_SC}. Then, $f$ is analytic in a neighbourhood of $z_1,\dots,z_m$. Moreover 
     \[
       f(z) = w_k+(z-z_k)^{\alpha_k}\tilde{f}_k(z)\qquad \forall k>m
     \]
     where $\tilde{f}_k(z)$ is analytic in a neighbourhood of $z_k$ and $\tilde{f}_k(z_k) \neq 0$.
\end{proposition}

\subsection{Schwarz-Christoffel map for general fractals}
In the general case we still choose $f$ as the conformal map from $\Omega_0$ to $\Omega_j$, but the choice of prevertices would depend on the geometry and symmetries of the fractal. Using the same arguments as in \eqref{eq:f_jder}, we have that
\[
  f(z) = w_k + (z-z_k)^{\alpha_k-\beta_k+1}\tilde f_k(z)
\]
with $\tilde f_k(z)$ a function analytic and non-zero in the neighbourhood of the prevertex $z_k$, $\alpha_k \pi$ the interior angle of $\Omega_j$ at corner $w_k$ and $\beta_k\pi$ the interior angle of $\Omega_0$ at $z_k$.  If $z_k$ is not a corner, $\beta_k = 1$ and we obtain the same behaviour as described in Proposition~\ref{prop:fj_smoothness} for the Koch snowflake.

In the subsequent analysis we will need to determine the behaviour of $|f'(z)|^{-1}$ near the prevertices. From the above, it follows that
\[
  |f'(z)|^{-1} \sim |z-z_k|^{\beta_k-\alpha_k} \text{ as } z \rightarrow z_k.
\]
In particular for $|f'(z)|^{-2}$ to be integrable in two dimensions, a condition we will require in the analysis, we need $\alpha_k-\beta_k < 1$. Note that this condition holds always if $\beta_k = 1$ as $\alpha_k \in (0,2)$. So, as long as we choose the conformal map in such a way that it sends corners of $\Omega_0$ to corners of $\Omega_j$ with $\alpha_k-\beta_k < 1$, the above is integrable. We formulate this as an assumption, which holds for the Koch snowflake construction described above.

\begin{assumptionb}
The conformal map $f : \Omega_0 \rightarrow \Omega_j$ is such that for each vertex $z_k$ of $\Omega_0$ mapped to a vertex $w_k$ of $\Omega_j$, we have $\alpha_k-\beta_k < 1$. Here $\alpha_k \pi$ is the interior angle of $\Omega_j$ at  $w_k$, and $\beta_k\pi$ is the interior angle of $\Omega_0$ at $z_k$. Note that this implies that
  \[
    \int_{\Omega_0} |f'(z)|^{-2} dz  < \infty.
    \]
\end{assumptionb}

For the Koch snowflake we fix the conformal map so that $\beta_k = \alpha_k$.

\subsection{Singularities of the eigenmodes}

%
%




Let $u:\Omega_j\longrightarrow \mathbb{R}$ be an eigenfunction of \eqref{DirichletLaplacian} associated with an eigenvalue $\omega^2$  for $\Omega=\Omega_j$. 

\begin{proposition}\label{prop:eigenmode_local}
 Let $(r, \theta)$ be the local polar coordinates  of $x \in \Omega_j$ with the origin at the vertex $w_k$.    Let $R >0$ be such that  $R < \min_{i \neq k}\dist(w_k, w_i)$ and $ R < \frac{\pi}{2|\omega|}$. Then for $r \in (0,R)$
\[
u(x) = \sum_{n = 1}^\infty a_n J_{\frac{n}{\alpha_k}}(|\omega|r) \sin \left(\frac{n\theta}{\alpha_k}\right),
\]
where
\[
a_n = \frac{2}{\alpha_k \pi J_{\frac{n}{\alpha_k}}(|\omega|R)} \int_0^{\alpha_k \pi} u(R,\theta) \sin \left(\frac{n\theta}{\alpha_k}\right) \, d\theta, \qquad n \in \mathbb{N}.
\]
\end{proposition}
\begin{proof}
 The proof is obtained in the usual way by  separation of variables.
\end{proof}

\begin{remark}
  The condition $R < \frac{\pi}{2|\omega|}$ ensures that $J_{\frac{n}{\alpha_k}}(|\omega|R)$ is non-zero. Further, we have
\[
\left|a_nJ_{\frac{n}{\alpha_k}}(|\omega|r)\right| \leq C\left(\frac{r}{R}\right)^{\frac{n}{\alpha_k}},
\]
giving absolute convergence of the series; see \cite{2007Chandler-WildeL}. 
\end{remark}


Now, let  $v = u \circ f  :\Omega_0\longrightarrow \mathbb{R}$
 be the transplanted eigenfunction for the eigenvalue $\omega^2$ of \eqref{eq:transplanted},  where $f: \Omega_0 \longrightarrow  \Omega_j$ is the Schwarz-Christoffel map \eqref{eq:fj_SC} in the case of the Koch snowflake.

 \begin{proposition}[Koch snowflake] \label{pro:singu}
 Let $(\varrho, \varphi)$ be the local polar coordinates of $y \in \Omega_0$ with the origin at the pre-vertex $z_k$.   Then
   \begin{enumerate}
   \item For $\Omega_j = \T_j$
\begin{align*}
  v(y) &= b_k(\varphi)\varrho^{3}+O(\varrho^{5}), & &\text{for } k = 1,2,3,\\
  v(y) &= b_k(\varphi)\varrho+O(\varrho^{5/3}), & &\text{for } k > 3 \text{ and } \alpha_k = 1/3,\\
  v(y) &= b_k(\varphi)\varrho+O(\varrho^2), & &\text{for } k > 3 \text{ and } \alpha_k = 4/3.
\end{align*}
\item and for $\Omega_j = \H_j$
\begin{align*}
  v(y) &= b_k(\varphi)\varrho^{3/2}+O(\varrho^{3}), & &\text{for } k = 1,\dots,6,\\
  v(y) &= b_k(\varphi)\varrho+O(\varrho^{2}), & &\text{for } k > 6 \text{ and } \alpha_k = 2/3,\\
  v(y) &= b_k(\varphi)\varrho+O(\varrho^2), & &\text{for } k > 6 \text{ and } \alpha_k = 5/3.
\end{align*}
\end{enumerate}
In the above, $b_k(\varphi)$ are analytic functions of $\varphi$ for $\varrho e^{i \varphi} \in \Omega_0$. These functions are different for each corner.
 \end{proposition}
 \begin{proof}
Recall the Maclaurin expansion of Bessel functions for $\nu>0$ \cite[10.2.2]{NIST},
 \begin{equation}
   \label{eq:j_nu_defn}
     J_\nu(z)=\left(\frac12 z\right)^\nu \sum_{\ell=0}^\infty (-1)^\ell\frac{(\frac14 z^2)^\ell}{\ell!\Gamma(\nu+\ell+1)}.   
 \end{equation}
If $k \leq m$, note that the singularity at $z_k$ of $v$ is the same as that of $u$, since $f$ is analytic near $z_k$ and $f(z_k) = z_k$. The result is then obtained from Proposition~\ref{prop:eigenmode_local} and the fact that $\nu$ is integer.

Next we give the details for $k > m$.   Combining Proposition~\ref{prop:fj_smoothness} and Proposition~\ref{prop:eigenmode_local}, it follows that 
   \[
     v(y) = v(z) = \sum_{n = 1}^\infty a_n J_{\frac{n}{\alpha_k}}(|\omega|r) \sin \left(\frac{n\theta}{\alpha_k}\right)
   \]
   near $z_k$,
   where
   \[
     r e^{i \theta} = \varrho^{\alpha_k}e^{i \alpha_k \varphi} \tilde{f}_k(z)
   \]
and the analytic function $\tilde{f}_k(z)$ is as in Proposition~\ref{prop:fj_smoothness}. In order to make use of the expansion \eqref{eq:j_nu_defn}, consider the terms of the form
\[
r^{n/\alpha_k+2\ell} = \varrho^{n+2\ell\alpha_k}|\tilde f_k(z)|^{n/\alpha_k+2\ell} \qquad n = 1,2,\dots \text{ and } \ell = 0,1,\dots.
\]
Note that $|\tilde f_k(z)|^{n/\alpha_k+2\ell}$ is  an analytic function of $\varrho$ and $\varphi$ in the vicinity of $z_k$ since $\tilde f_k(z_k) \neq 0$. Further
\[
  \theta = \alpha_k \varphi+\operatorname{Arg} \tilde{f}_k(z).
\]
The result is obtained by isolating the leading term in each one of the cases. 
 \end{proof}


Two remarks are now in place.

As a consequence of this proposition, in the case of the Koch snowflake, it follows that the strongest singularity for $\T_j$ is near the angles $\alpha_k = 1/3$ and for $\H_j$ near the 6 original corners. Therefore, overall, the transplantation has reduced the strongest singularity from $r^{3/4}$ to $\varrho^{5/3}$ for $\T_j$ and from $r^{3/5}$ to $\varrho^{3/2}$ for $\H_j$. This implies, for example,  that the first derivative of $v$ is bounded but the first derivative of $u$ is unbounded.

In the case of the $\T_j$ polygons, the original eigenfunction $u$ is analytic near the three vertices $w_k$, $k = 1,2,3$, and the same holds for the transplanted eigenfunction near the corresponding prevertices $z_k$.

\section{Formulation as a system  \label{main3}}

In this section we set the theoretical framework of the block W3).
For this purpose we define a selfadjoint operator $\overline{\mathcal{T}_y}$ of order 1 associated with the eigenvalue problem \eqref{eq:transplanted}, whose squared non-zero spectrum coincides with the eigenvalues of \eqref{DirichletLaplacian}. In \S\ref{main4} we will formulate a procedure for computing lower and upper bounds for $\operatorname{spec}(\overline{\mathcal{T}_y})$, which involves the square of this operator. For this, the trial functions are required to lie in the operator domain of $\overline{\mathcal{T}_y}$ (the form domain of $\overline{\mathcal{T}_y}^2$). In \S\ref{sec:trans_op} we describe explicitly an operator core $\mathcal{D}$ in terms of the derivative of the conformal map $|f'|$.

\subsection{The div-grad operator}
Let 
\[
  \overbrace{
  \begin{bmatrix}
   0 & i \div_x \\ \\ i \grad_x & 0 
  \end{bmatrix}}^{\mathcal{G}_x}:\overbrace{\begin{matrix}   \mathsf{H}^1_0(\Omega_j) \\ \times \\ \mathsf{H}(\div,\Omega_j)
  \end{matrix}}^{\operatorname{D}(\mathcal{G}_x)} \longrightarrow \overbrace{\begin{matrix} \mathsf{L}^2(\Omega_j) \\ \times \\ \mathsf{L}^2(\Omega_j)^2
  \end{matrix}}^{\mathsf{L}^2(\Omega_j)^3}.
\]  
The densely defined operator $\mathcal{G}_x:\operatorname{D}(\mathcal{G}_x)\longrightarrow \mathsf{L}^2(\Omega_j)^3$ is selfadjoint, because the adjoint of the minimal operator $i\grad_x: \mathsf{H}^1_0(\Omega_j)\longrightarrow \mathsf{L}^2(\Omega_j)^2$ is the maximal operator $i \div_x:\mathsf{H}(\div,\Omega_j)\longrightarrow \mathsf{L}^2(\Omega_j)$ and vice versa.

Denote the selfadjoint operator associated to \eqref{DirichletLaplacian} on $\Omega=\Omega_j$ by
\[
    -\Delta_x:\operatorname{D}(\Delta_x)\longrightarrow L^2(\Omega_j).
\]
Here the domain of the Dirichlet Laplacian is defined via
von Neumann's Theorem \cite[p.275]{1967Kato}, as
\[
    \operatorname{D}(\Delta_x)=\{u\in H^1_0(\Omega_j):\grad u\in 
    \mathsf{H}(\div,\Omega_j)\}\subset L^2(\Omega_j).
\]
See Appendix~\ref{specmat}.

\begin{lemma} \label{eigenvaluesofsystem}
 The vector $\begin{bmatrix}  u \\  \underline{s}  \end{bmatrix} \in \operatorname{D}(\mathcal{G}_x)$ is an eigenfunction of $\mathcal{G}_x$ if and only if, 
\begin{enumerate}
\item \label{eigenvaluesofsystem1}  either $u\in \operatorname{D}(\Delta_x)$, $-\Delta_x u = \omega^2 u$  and $\underline{s}=\frac{\pm i}{|\omega|}\grad_x u$
\item \label{eigenvaluesofsystem2} or $u=0$ and $\div_x \underline{s}=0$. 
\end{enumerate}
Moreover, $\begin{bmatrix}  u \\  \underline{s}  \end{bmatrix}$ is associated to the eigenvalue $\pm \omega$ in the case \ref{eigenvaluesofsystem1} and to the eigenvalue $0$ in the case \ref{eigenvaluesofsystem2}. 
\end{lemma}
\begin{proof} See Lemma~\ref{zeroinpointspecofEandTTstar} and the proof of Lemma~\ref{pointspecofEandTTstar}. \end{proof}

Denote by $\{u_k\}_{k\in \mathbb{N}}\subset \operatorname{D}(\Delta_x)$ an orthonormal basis of eigenfunctions such that $-\Delta_x u_k=\omega_k^2 u_k$.  
As a consequence of Lemma~\ref{eigenvaluesofsystem}, the family
\[
     \mathcal{E}=\left\{ \begin{bmatrix} u_k \\  \pm \underline{s}_k \end{bmatrix}\right\}_{k\in \mathbb{N}}\cup
   \left\{ \begin{bmatrix} 0 \\  \underline{\sigma}_n \end{bmatrix}\right\}_{n\in \mathbb{N}}
\]
where $\underline{s}_k=\frac{\pm i}{|\omega_k|}\grad_x u_k$ and we pick $\{\underline{\sigma}_n\}_{n=1}^\infty\subset H(\div,\Omega_j)$ an orthonormal basis of $\ker(\div)$,
is a complete family of eigenfunctions of $\mathcal{G}_x$. Moreover
\[
    \operatorname{spec}(\mathcal{G}_x)=\{\pm \omega_k(\Omega_j),0\}. 
\]
In fact $\mathcal{E}$ is an orthonormal basis of $L^2(\Omega_j)^3$. Each non-zero eigenvalue is discrete and the eigenvalue zero is degenerate (infinite multiplicity).

\subsection{The transplanted selfadjoint operator}\label{sec:trans_op}
Let 
\[
    \mathcal{D}:=\left\{\begin{bmatrix} \tilde{v} \\ \underline{t} \end{bmatrix}\in L^2(\Omega_0)^3:
     |f'|^{-1}\tilde{v}\in H^1_0(\Omega_0),\, |f'|^{-1} \div_y \underline{t}\in L^2(\Omega_0)      
\right\}
\]
and define
\[
     \mathcal{T}_y=
     \begin{bmatrix}0 & i|f'|^{-1}\div_y \\ i\grad_y |f'|^{-1} & 0 \end{bmatrix}:
      \mathcal{D}\longrightarrow L^2(\Omega_0)^3.
\]
Then $\mathcal{T}_y$ is a densely defined symmetric operator.

\begin{theorem} \label{operatorformulation}
The operator $(\mathcal{T}_y,\mathcal{D})$ on $L^2(\Omega_0)^3$ has an orthonormal basis of eigenfunctions in its domain. The closure
\[
      \overline{\mathcal{T}_y}: \operatorname{D}(\overline{\mathcal{T}_y})\longrightarrow L^2(\Omega_0)^3
\]
is selfadjoint. Moreover,
\[
     \operatorname{spec}(\overline{\mathcal{T}_y})=\operatorname{spec}(\mathcal{G}_x)=\{\pm w_k(\Omega_j),0\}.
\]
\end{theorem}

The remainder of this section is devoted to the proof of this theorem. Our first task will be to verify that the transplanted eigenfunctions are in the domain of $\mathcal{T}_y$. Let 
\[
v_k=u_k\circ f, \qquad \tilde{v}_k=|f'|v_k, \qquad
\underline{t}_k=(\nabla_y f)^T\underline{s}_k\circ f \qquad \text{and} \qquad \underline{\tau}_n=(\nabla_y f)^T\underline{\sigma}_n\circ f.
\] 

\begin{lemma} \label{functionsindomainofT}
\[
     \tilde{\mathcal{E}}=\left\{\begin{bmatrix} \tilde{v}_k \\ \pm \underline{t}_k \end{bmatrix} ,
 \begin{bmatrix} 0 \\  \underline{\tau}_n \end{bmatrix}\right\}_{k,n\in \mathbb{N}} \subset \mathcal{D}.
\] 
\end{lemma}
\begin{proof}
Let us first show that
\[
    \begin{bmatrix} \tilde{v}_k \\ \pm \underline{t}_k \end{bmatrix}\in \mathcal{D}.
\] From \eqref{honezero} with $u=u_k$ and $v=v_k$, it follows that 
\[
    \int_{\Omega_0}|\nabla _y v_k|^2 dy<\infty.
\]
Since $\Omega_0$ is compact, by Sobolev embedding it then follows that also
\[
    \int_{\Omega_0}|v_k|^2 dy<\infty
\] 
and $|f'|^{-1}\tilde{v}_k=v_k\in H^1_0(\Omega_0)$. This is the first condition in the definition of $\mathcal{D}$.

Now the second condition. Since 
\[
    \int_{\Omega_0} |\underline{t}_k|^2 dy=\int_{\Omega_0}\left[(\nabla_yf)(\nabla_yf)^T(\underline{s}_k\circ f)\right]\cdot (\underline{s}_k\circ f)dy=\int_{\Omega_j} |\underline{s}_k|^2 dx 
\]
we gather that $\underline{t}_k\in L^2(\Omega_0)^2$. Then
\begin{align*}
  \div_y \underline{t}_k&=(\underline{s}_k\circ f)\cdot \underline{\Delta} f +
\operatorname{Tr} \left((\nabla_y f)^T \nabla_y (\underline{s}_k\circ f)  \right) \\
&= 0 + 
\operatorname{Tr} \left((\nabla_y f)^T [(\nabla_y \underline{s}_k)^T\circ f] (\nabla_y f)   \right) \\
&= |f'|^2 \div_x \underline{s}_k\circ f.
\end{align*}
Hence
\[
   |f'|^{-1}  \div_y \underline{t}_k =|f'| \div_x \underline{s}_k\circ f,
\]
so
\[
    \int_{\Omega_0}\left| |f'|^{-1}  \div_y \underline{t}_k   \right|^2 dy=
\int_{\Omega_0}|f'|^{2}  \left|\div_x \underline{s}_k\circ f   \right|^2 dy=
    \int_{\Omega_j}\left| \div_x \underline{s}_k   \right|^2 dx<\infty.
\]
This is the second condition in the definition of $\mathcal{D}$.

It is only left to show that
\[
    \begin{bmatrix} 0\\ \pm \underline{\tau}_n \end{bmatrix}\in \mathcal{D}.
\]
On the one hand,
\[
    \int_{\Omega_0} |\underline{\tau}_n|^2 dy=\int_{\Omega_0}\left[(\nabla_yf)(\nabla_yf)^T(\underline{\sigma}_n\circ f)\right]\cdot (\underline{\sigma}_n\circ f)dy=\int_{\Omega_j} |\underline{\sigma}_n|^2 dx. 
\]
On the other hand,
\[
     \div_y \underline{\tau}_n =|f'|^2 \div_x \underline{\sigma}_n\circ f=0\in L^2(\Omega_0).
\]
\end{proof}

The family  $\tilde{\mathcal{E}}$ in this lemma is a family of eigenfunctions of $\mathcal{T}_y$. Indeed
\begin{align*}
     \mathcal{T}_y \begin{bmatrix} \tilde{v}_k \\ \pm \underline{t}_k \end{bmatrix} & =  \begin{bmatrix}  \pm i  |f'|^{-1}\div_y \underline{t}_k \\ i\grad_y v_k  \end{bmatrix} = 
\begin{bmatrix} \pm i  |f'|(\div_x \underline{s}_k)\circ f \\ i(\nabla_y f)^T(\grad_x u_k)\circ f  \end{bmatrix}
\\ &= \begin{bmatrix} |f'| & 0 \\ 0 & (\nabla_y f)^T \end{bmatrix} \mathcal{G}_x\begin{bmatrix}u_k \\ \pm \underline{s}_k \end{bmatrix} \circ f \\
 & = \pm \omega_k \begin{bmatrix} |f'|u_k\circ f \\ \pm (\nabla_y f)^T \underline{s}_k\circ f \end{bmatrix}= \pm \omega_k \begin{bmatrix} \tilde{v}_k \\ \pm \underline{t}_k \end{bmatrix}
\end{align*}
and
\[
   \mathcal{T}_y \begin{bmatrix} 0 \\  \underline{\tau}_n \end{bmatrix} =
\begin{bmatrix} i  |f'|(\div_x \underline{\sigma}_n)\circ f \\  0 \end{bmatrix}=0.
\] 
In fact it is a complete family of eigenfunctions as we shall see next.

\begin{lemma}
\[
    \operatorname{Span} \tilde{\mathcal{E}}= L^2(\Omega_0)^3.
\]
\end{lemma}
\begin{proof}
We verify that $\tilde{\mathcal{E}}^{\perp}=\{0\}$.
Suppose that 
\begin{equation} \label{vtperptoeverybody}
\int_{\Omega_0} \begin{bmatrix} \tilde{v}_k \\ \pm \underline{s}_k \end{bmatrix}\cdot \begin{bmatrix} v \\ \underline{t} \end{bmatrix} dy=0=\int_{\Omega_0} \begin{bmatrix} 0 \\ \underline{\tau}_n \end{bmatrix}\cdot \begin{bmatrix} v \\ \underline{t} \end{bmatrix} dy \qquad \forall k,n\in\mathbb{N}.
\end{equation}
Let $g=f^{-1}:\mathbb{R}^2 \rightarrow \mathbb{R}^2$ be the inverse map to $f$. Then $u=v\circ g$ and $\underline{s}=\underline{t}\circ g$ and
\begin{align*}
0&=\int_{\Omega_0} \begin{bmatrix} |f'|u_k\circ f \\ \pm (\nabla_yf)^T\underline{s}_k\circ f \end{bmatrix}\cdot \begin{bmatrix} u\circ f \\ \underline{s}\circ f \end{bmatrix} dy =\int_{\Omega_0} \begin{bmatrix} u_k\circ f \\ \pm \underline{s}_k\circ f \end{bmatrix}\cdot \begin{bmatrix} |f'| u\circ f \\ (\nabla_yf) \underline{s}\circ f \end{bmatrix} dy \\
&=\int_{\Omega_j} \begin{bmatrix} u_k \\ \pm \underline{s}_k \end{bmatrix}\cdot \begin{bmatrix} |g'| u \\ |g'|^2((\nabla_yf)\circ g) \underline{s} \end{bmatrix} dx
\end{align*}
for all $k\in \mathbb{N}$. Further
\[
    0=\int_{\Omega_j} |g'|^2 \underline{\sigma}_n\cdot (\nabla_yf \circ g)\underline{s}\ dx
\]
for all $n\in \mathbb{N}$. Since $\mathcal{E}$ is an orthonormal basis of $L^2(\Omega_j)^3$, then 
\[
    |g'|u=0 \qquad \text{and} \qquad |g'|^2(\nabla_y f\circ g)\underline{s}=0.
\]
Hence, since $|g'|\not=0$ a.e. and $\det\nabla_y f\circ g=|g'|^{-2}\not=0$ a.e., $u=0$ and $\underline{s}=0$. Thus \eqref{vtperptoeverybody} implies $\begin{bmatrix} v \\ \underline{t} \end{bmatrix}=0$.
\end{proof}

In order to generate an orthonormal family of eigenfunctions apply Gram-Schmidt to $\tilde{\mathcal{E}}$ which might not be orthonormal a priori. Note that in fact $\mathcal{T}_y$ is essentially selfadjoint, \cite[Lemma~1.2.2]{1995Davies}. This completes the proof of Theorem~\ref{operatorformulation}.

 \begin{remark}
 Since $|f'|$ has singularities on $\partial \Omega_0$, it is not  a priori clear whether $(\mathcal{T}_y,\mathcal{D})$ is  closed. This is a rather subtle  point. We are unaware of any investigation in this respect.
 \end{remark}

\section{Computation of the upper and lower bounds\label{main4}}
We now describe one possible method to determine bounds for the eigenvalues of the operator $\overline{\mathcal{T}_y}$ for the block W4).
We have chosen the quadratic method \cite{2004Levitin} which fully avoids spectral pollution \cite{2000Shargorodsky} and is shown to be reliable for computing eigenvalues. For a full list of references see \cite[Section~6.1]{2016Boulton}. For alternative approaches see \cite{2012MarSchei,SebV,CC,2017BBB}.

\subsection{The quadratic method}
Given a subspace $\mathcal{L}\subset \operatorname{D}(\overline{\mathcal{T}_y})$ of dimension $d<\infty$, the second order spectrum \cite{1998Davies} of $\overline{\mathcal{T}_y}$ relative to $\mathcal{L}$ is the spectrum of the following quadratic matrix polynomial weak eigenvalue problem. 
 
\begin{problem} \label{secondorderspectrum} Find $\lambda\in \mathbb{C}$ and $0\not=\begin{bmatrix} v \\ \underline{t} \end{bmatrix}\in \mathcal{L}$ such that
\[
   \left \langle (\overline{\mathcal{T}_y}-\lambda)  \begin{bmatrix} v \\ \underline{t} \end{bmatrix}, (\overline{\mathcal{T}_y}-\lambda^*)  \begin{bmatrix} \tilde{v} \\ \underline{\tilde{t}} \end{bmatrix} \right \rangle =0 \qquad \qquad \forall \begin{bmatrix} \tilde{v} \\ \underline{\tilde{t}} \end{bmatrix}\in \mathcal{L}.
\]
\end{problem}

Given a basis for the subspace $\mathcal{L}$,
\[
  \mathcal{L}=\operatorname{span}\{\underline{b}_j\}_{j=1}^d,
\]
and writing
\[
  \begin{bmatrix} v \\ \underline{t}\end{bmatrix}=\sum_{j=1}^d \alpha_j \underline{b}_j \qquad \text{for} \qquad \underline{\alpha}=(\alpha_j)_{j=1}^d\in \mathbb{C}^d,
\]
this problem becomes equivalent to
\[
Q(\lambda) \underline{\alpha} =0 \qquad \text{for} \qquad Q(z)=K-2z L +z^2 M,
\]
where
\[
   K=[\langle \overline{\mathcal{T}_y}\underline{b}_j, \overline{\mathcal{T}_y}\underline{b}_k\rangle]_{jk=1}^d
\qquad L=[\langle  \overline{\mathcal{T}_y}\underline{b}_j,\underline{b}_k\rangle]_{jk=1}^d
\qquad M=[\langle \underline{b}_j,\underline{b}_k\rangle]_{jk=1}^d.
\]
The $\lambda\in \mathbb{C}$ solutions to Problem~\ref{secondorderspectrum} are therefore the spectrum of the quadratic matrix polynomial $Q(z)$. Since $Q(z)$ is selfadjoint, this set is symmetric with respect to the real line. Since $\det M\not=0$, it consists of at most $2d$ distinct isolated points.

The following relation between the second order spectra and the spectrum of $\overline{\mathcal{T}_y}$ is crucial below. 
Let
\[
   \mathbb{D}(a,b)=\left\{ z\in \mathbb{C}:\left|z-\frac{a+b}{2}\right|< 
\frac{b-a}{2} \right\}.
\]
Then,
\begin{equation} \label{strauss-boulton-levitin}
   \left.\begin{aligned}
   &(a,b)\cap \operatorname{spec} \overline{\mathcal{T}_y} =\{\omega\}
\\  & \det Q(\lambda)=0
\\ & \lambda \in \mathbb{D}(a,b)
\end{aligned}
\right\}
\Rightarrow \re \lambda -\frac{|\im \lambda|^2}{b-\re \lambda}
<\omega <  \re \lambda +\frac{|\im \lambda|^2}{\re \lambda-a}
\end{equation}
 See \cite[Remark~2.3]{2011Strauss} and  \cite[Corollary~2.6]{2007Boulton}.

\subsection{Finite element approximation of the eigenvalue bounds}

We now show a possible concrete family of subspaces $\mathcal{L}$.
Let $\Xi_h$ be a uniform triangulation of  $\Omega_{0}$, define the corresponding space of piecewise polynomials to be
\begin{equation} \label{finiteelementspaces}
   \hat{\mathcal{L}}=\left\{\begin{bmatrix} v \\ \underline{t} \end{bmatrix}
    \in C^0(\overline{\Omega_{0}})^3:\begin{bmatrix} v|_{K} \\ \underline{t}|_K \end{bmatrix} \in
   \mathbb{P}_p(K)^3\ \ \forall K\in \Xi_h,\, v|_{\partial \Omega_{0}}=0 \right\},
\end{equation}
and let 
\[
     F=\begin{bmatrix} |f'| & 0 & 0  \\ 0 & 1 & 0\\ 0 & 0 & 1 \end{bmatrix}.
\]
Consider the following reformulation of Problem~\ref{secondorderspectrum}.

\begin{problem}  \label{stableproblem} 
Find $\lambda$ and  $0\not=\begin{bmatrix} v \\ \underline{t} \end{bmatrix}\in \hat{\mathcal{L}}$ such that for all $\begin{bmatrix} \tilde{v} \\ \underline{\tilde{t}} 
    \end{bmatrix}\in \hat{\mathcal{L}}$
\[
    \left\langle F^{-1} \mathcal{G}_y \!
    \begin{bmatrix} v \\ \underline{t} \end{bmatrix}
    ,F^{-1} \mathcal{G}_y \!
 \begin{bmatrix} \tilde{v} \\ \underline{\tilde{t}} \end{bmatrix}
     \right\rangle
     -2\lambda
     \left\langle \mathcal{G}_y \!
    \begin{bmatrix} v \\ \underline{t} \end{bmatrix}
    , \begin{bmatrix} \tilde{v} \\ \underline{\tilde{t}} 
    \end{bmatrix}
     \right\rangle 
     +\lambda^2 \left\langle F \!
    \begin{bmatrix} v \\ \underline{t} \end{bmatrix}
    ,F\!
 \begin{bmatrix} \tilde{v} \\ \underline{\tilde{t}} \end{bmatrix}
     \right\rangle=0.
\]
\end{problem}

The substitution $\begin{bmatrix} v \\ \underline{t} \end{bmatrix}=
F^{-1}\begin{bmatrix} w \\ \underline{r} \end{bmatrix}$ and
$\begin{bmatrix} \tilde{v} \\ \underline{\tilde{t}} \end{bmatrix}=
F^{-1}\begin{bmatrix} \tilde{w} \\ \underline{\tilde{r}} \end{bmatrix}$
yields an equivalence between problems~\ref{stableproblem} and \ref{secondorderspectrum}, where the subspaces are deformed by the action of $|f'|$, 
\begin{equation} \label{subspaceL}
      \mathcal{L}=F\hat{\mathcal{L}}.
\end{equation}
Indeed Problem~\ref{stableproblem} is equivalent to finding $\lambda$ and  $0\not=\begin{bmatrix} w \\ \underline{r} \end{bmatrix}\in \mathcal{L}$ such that 
\[
    \left\langle \overline{\mathcal{T}_y}\!
    \begin{bmatrix} w \\ \underline{r} \end{bmatrix}
    ,\overline{\mathcal{T}_y}\!
 \begin{bmatrix} \tilde{w} \\ \underline{\tilde{r}} \end{bmatrix}
     \right\rangle
     -2\lambda
     \left\langle \overline{\mathcal{T}_y} \!
    \begin{bmatrix} w \\ \underline{r} \end{bmatrix}
    , \begin{bmatrix} \tilde{w} \\ \underline{\tilde{r}} 
    \end{bmatrix}
     \right\rangle 
     +\lambda^2 \left\langle
    \begin{bmatrix} w \\ \underline{r} \end{bmatrix}
    ,
 \begin{bmatrix} \tilde{w} \\ \underline{\tilde{r}} \end{bmatrix}
     \right\rangle=0 \qquad \forall\begin{bmatrix} \tilde{w} \\ \underline{\tilde{r}} 
    \end{bmatrix}\in \mathcal{L}.
\]
The latter is exactly Problem~\ref{secondorderspectrum} for $\mathcal{L}$ given by \eqref{subspaceL}. For the quadratic method to be free from spectral pollution we require $\mathcal{L} \subset \operatorname{D}(\overline{\mathcal{T}}_y)$. As we shall see next, this is indeed the case.

\begin{lemma} \label{twistedfiniteelementspaceindomain}
Let $\hat{\mathcal{L}}$ be given by \eqref{finiteelementspaces} and
$\mathcal{L}$ be given by \eqref{subspaceL}. Then
$
    \mathcal{L}\subset \mathcal{D}.
$
\end{lemma}
\begin{proof}
Let $\begin{bmatrix}\tilde{v} \\ \underline{t} \end{bmatrix}\in \hat{\mathcal{L}}$. Since
$\hat{\mathcal{L}}\subset H^1_0(\Omega_0) \times
H^1(\Omega_0)^2$, then $\tilde{v}\in H^1_0(\Omega_0)$ and $\underline{t}\in H^1(\Omega_0)^2\subset H(\div,\Omega_0)$. As the first entry of
$
    F\begin{bmatrix}\tilde{v} \\ \underline{t} \end{bmatrix}
$   
is $|f'|\tilde{v}$, it indeed satisfies the first condition in the definition of $\mathcal{D}$. 

Now, the second entry of $
    F\begin{bmatrix}\tilde{v} \\ \underline{t} \end{bmatrix}
$
is $\underline{t}$. Since $\underline{t}$ is continuous and  piecewise polynomial its divergence is bounded on $\overline{\Omega_0}$. Hence
\[
   \int_{\Omega_0}\left| |f'|^{-1}  \div_y\underline{t}\right|^2 dy\leq
   \| \div_y\underline{t}\|_{\infty}^2  \int_{\Omega_0} |f'|^{-2} dy.
\] 
The function $|f'|^{-1}$  has singularities only on $\partial \Omega_0$. According to Assumption~B, 
\[ 
\int_{\Omega_0} |f'|^{-2} dy< \infty.
\] 
Then, indeed $|f'|^{-1}\div_y\underline{t}\in L^2(\Omega_0)$, ensuring the second condition in the definition of $\mathcal{D}$. 
\end{proof}

\begin{remark} Let
\[
    \begin{bmatrix} \tilde{v}_1 \\ \underline{t}_1 \end{bmatrix}\in\tilde{\mathcal{E}}
\]
be the normalised eigenfunction associated to the first positive eigenvalue
\[\omega_1(\Omega_j)\in \operatorname{spec}(\overline{\mathcal{T}_y}). \]
A convergence analysis of the finite element method at individual $\Omega_j$ can be carried out in the context of  \cite[Theorem~3.2]{2016Boulton2}. It shows that if there exists a constant $c_1 > 0$ so that 
\[
\inf_{\begin{bmatrix} w_h & \underline{r}_h \end{bmatrix}^T\in \mathcal{L}}   \left\| \begin{bmatrix} \tilde{v}_1 \\ \underline{t}_1 \end{bmatrix}
   -\begin{bmatrix} w_h \\ \underline{r}_h \end{bmatrix} \right\|_{L^2(\Omega_0)^3}
   +\left\|\mathcal{T}_y\left(
   \begin{bmatrix} \tilde{v}_1 \\ \underline{t}_1 \end{bmatrix}-
   \begin{bmatrix} w_h \\ \underline{r}_h \end{bmatrix}\right)
   \right\|_{L^2(\Omega_0)^3}< c_1 h^p,
\]
then there exists $\lambda_h$ such that $\det Q(\lambda_h)=0$ and
\[
       |\lambda_h - \omega_1|<c_2 h^{p/2}.
\]
The hypothesis translates into the subspace $\hat{\mathcal{L}}$ as follows,
\[
  \begin{split}    
\inf_{\begin{bmatrix} v_h & \underline{r}_h \end{bmatrix}^T\in \hat{\mathcal{L}}}        \left\|  \begin{bmatrix} |f'|v_1 \\ \underline{t}_1 \end{bmatrix}
   \right.&-\left.\begin{bmatrix} |f'|v_h \\ \underline{r}_h \end{bmatrix} \right\|_{L^2(\Omega_0)^3}
   \\&+\left\|F^{-1}\mathcal{G}_y\left(
   \begin{bmatrix} |f'|v_1 \\ \underline{t}_1 \end{bmatrix}-
   \begin{bmatrix} |f'|v_h \\ \underline{r}_h \end{bmatrix}\right)
   \right\|_{L^2(\Omega_0)^3}< c_1 h^p.
  \end{split}
\]
 Here $v_1=u_1\circ f$.
These convergence estimates might be investigated in future work.
\end{remark}

\section{Computations for a Koch snowflake \label{calculations}}
In this final section we show a particular implementation of the workflow~W1)-W4) for $\Sigma\subset \mathsf{D}$ a Koch snowflake such that $\partial \Sigma$ is inscribed in $\partial \mathsf{D}$ the unit circle. This implementation leads to \eqref{finalbounds}.
In \cite[Table~2]{1999Cureton} an approximation of the ground eigenvalue for the hexagon was reported as $\omega_1^2(\mathsf{H}_0)\leq 7.155339146$. Later in \cite[Table 2]{2007Banjai} numerical evidence was given indicating that $\omega_1^2(\Sigma)= 13.116184\overline{3}$  with doubt over the last digit. This number is within the estimate \eqref{finalbounds}.

In block W1) we set $j=1,\ldots,10$ where $\T_j$ and $\H_j$ are chosen as in \S~\ref{classsnofla}. We find upper and lower bounds for $\omega_1^2(\T_j)$ and $\omega_1^2(\H_j)$. By domain monotonicity \eqref{domainmonotonicity} lower bounds for $\omega_1^2(\H_j)$ are lower bounds for $\omega_1^2(\Sigma)$ and upper bounds for $\omega_1^2(\T_j)$ are upper bounds for $\omega_1^2(\Sigma)$. We derive \eqref{finalbounds} from
\[
      \omega_1^2(\H_{10})_{\mathrm{lower}}<\omega_1^2(\Sigma)<
   \omega_1^2(\T_{10})^{\mathrm{upper}}
\]
using the numerical estimates in the last row of Table~\ref{table1}.

We compute the conformal maps $f_j$ for block W2) by means of the highly accurate procedure described in \cite{2003Banjai,2008Banjai} coded in C++.  The Schwarz-Christoffel formula \eqref{eq:fj_SC} is semi-explicit as the position of the pre-vertices $z_k$ for $j > 0$ is not a-priori known and needs to be computed as the solution of a non-linear parameter problem. Using a simple iteration due to Davis \cite{Davis} and accelerating the computation using the fast multipole method \cite{Gree}, we solve this problem for hundreds of thousands of pre-vertices. The details and required modifications to the original algorithms are given in \cite{2008Banjai,2003Banjai}.

For block W4), we solve Problem~\ref{stableproblem}.  We pick Lagrange elements of order $p=5$ in \eqref{finiteelementspaces} on a mesh for $\Omega_0$ made of uniformly distributed equilateral triangles of identical area. We start with an initial mesh for $\H_0$ made of 6 equilateral triangles and $\T_0$ made of 4 equilateral triangles. Then refine each mesh a number of times. In each refinement the number of elements is multiplied by 4. 

The ground eigenvalue on the unit disk is $\omega_1^2(\mathsf{D})=\mathrm{j}_{0,1}$ and $14.68 < \mathrm{j}_{1,1}=\omega_2^2(\mathsf{D})$. To get lower bounds for $\omega_1(\Omega_j)$, we appeal to \eqref{domainmonotonicity} and consider \eqref{strauss-boulton-levitin} fixing $a=0$ and $b\leq \omega_{2}(\mathsf{D})$  known.  In practice we choose
\[
    b< \sqrt{j_{1,1}}.
\]
We formulated and solved Problem~\ref{stableproblem} numerically using the commercial package Comsol Multiphysics and ran the simulations in Comsol Livelink for Matlab.

\subsection{Our best estimate}

In Table~\ref{table1} we show our computation of $\omega_{1}^2(\T_{10})$ and $\omega_1^2(\H_{10})$, as we refine the original mesh the indicated number of times. 

For $\T_{10}$ accuracy stalls from the 6th to the 7th refinement, then it jumps by a considerable margin. We believe that this phenomenon is linked to the structure of the eigenfunction for $\T_{10}$ near the boundary, but we can say no more at present. Similarly a stall in accuracy occurs between the 3rd and 5th refinement for $\H_{10}$.

We stopped the calculation for $\T_{10}$ after 10 refinements and for $\H_{10}$ after 7 refinements. Rounding error and lack of computer memory, perhaps due to the integrator coded in the commercial package, took over after this.  

\begin{table}[t]
\centerline{\begin{tabular}{|l|c|}
  \hline
 $\omega^2_1(\mathsf{H}_{10})_{\mathrm{lower}}^{\mathrm{upper}}$ &  Refinement \\ 
  \hline
$13.1_{04}^{17}$&     2\\
$13.11_{49}^{61}$&    3\\
$13.11_{57}^{61}$&    4\\
$13.11_{59}^{61}$&    5\\
$13.1160_{02}^{16}$&  6\\
$13.11601_{15}^{20}$& 7 \\
  \hline
  \end{tabular} \hspace{5mm}
\begin{tabular}{|l|c|}
  \hline
 $\omega^2_1(\mathsf{T}_{10})_{\mathrm{lower}}^{\mathrm{upper}}$ &  Refinement \\ 
  \hline
$13._{09}^{12}$ &        4\\
$13.1_{03}^{17}$ &       5\\
$13.11_{26}^{65}$ &      6\\
$13.11_{57}^{63}$ &      7\\
$13.116_{16}^{24}$ &     8\\
$13.1162_{08}^{29}$ &    9 \\
$13.11622_{10}^{76}$ &   10 \\
  \hline
  \end{tabular}
}
\caption{For level $j=10$. Upper and lower bounds for the ground eigenvalue on $\mathsf{H}_{10}$  and $\mathsf{T}_{10}$. The mesh is made of equilateral triangles. At each refinement we increase the number of triangles by a factor of 4. \label{table1}}
\end{table}

\subsection{The optimal rate of interior-exterior domain approximation}
In order to test optimality of the decreasing rate of
\[
\omega_1^2(\T_j)-\omega_1^2(\H_j)
\]
established in Lemma~\ref{lem:def_sf}, we present in Table~\ref{table2} computation of $\omega_1^2(\Omega_j)$ with the shown number of refinements for $j=0,\ldots,10$. 
 Note that \[\omega_1^2(\mathsf{T}_0)=\frac{16\pi^2}{9}\leq 17.54597.\] Therefore the lower bound for level $j=0$ is not given, because the $b$ chosen in \eqref{strauss-boulton-levitin} is not below $\omega_1(\T_{0})$.

\begin{table}
\center{\begin{minipage}{.49\textwidth}{\tiny \begin{tabular}[b]{|c|lc||lc|}
  \hline
   $j$ & $\omega^2_1(\mathsf{H}_j)_{\mathrm{lower}}^{\mathrm{upper}}$ & R &$\omega^2_1(\mathsf{T}_j)_{\mathrm{lower}}^{\mathrm{upper}}$ &  R \\ 
  \hline
0& $\hspace{.5em}7.15533_{83}^{94}$&    4 &  $17.5459633^{80}$&     4\\   
1& $11.78144_{19}^{39}$&   5&              $13.402_{24}^{73}$&     5\\  
2& $12.51986_{72}^{87}$&    6 &              $13.268_{56}^{86}$&   6\\  
3& $12.89778_{06}^{23}$&    7 &             $13.170_{69}^{77}$&    7 \\  
4& $13.03710_{57}^{92}$&      7  &                    $13.1357_{33}^{54}$&  8 \\ 
5& $13.0876 _{89}^{93}$&       7   &                   $13.123_{19}^{22}$&  8 \\             
6& $13.10593_{52}^{82}$&          7 &                 $13.118_{68}^{71}$&  8  \\            
7& $13.112  _{49}^{51}$&            7 &                $13.1170_{73}^{93}$&  8\\             
8& $13.1148 _{59}^{63}$&           7  &                $13.116_{49}^{52}$&  8 \\            
9& $13.11570_{73}^{99}$&            7  &                $13.116_{28}^{31}$ & 9  \\
10 & $13.11601_{15}^{20}$ &  7 & $13.11622_{10}^{76}$ &    10  \\
  \hline
  \end{tabular}} \end{minipage} \begin{minipage}{.49\textwidth}\includegraphics[width=.99\textwidth]{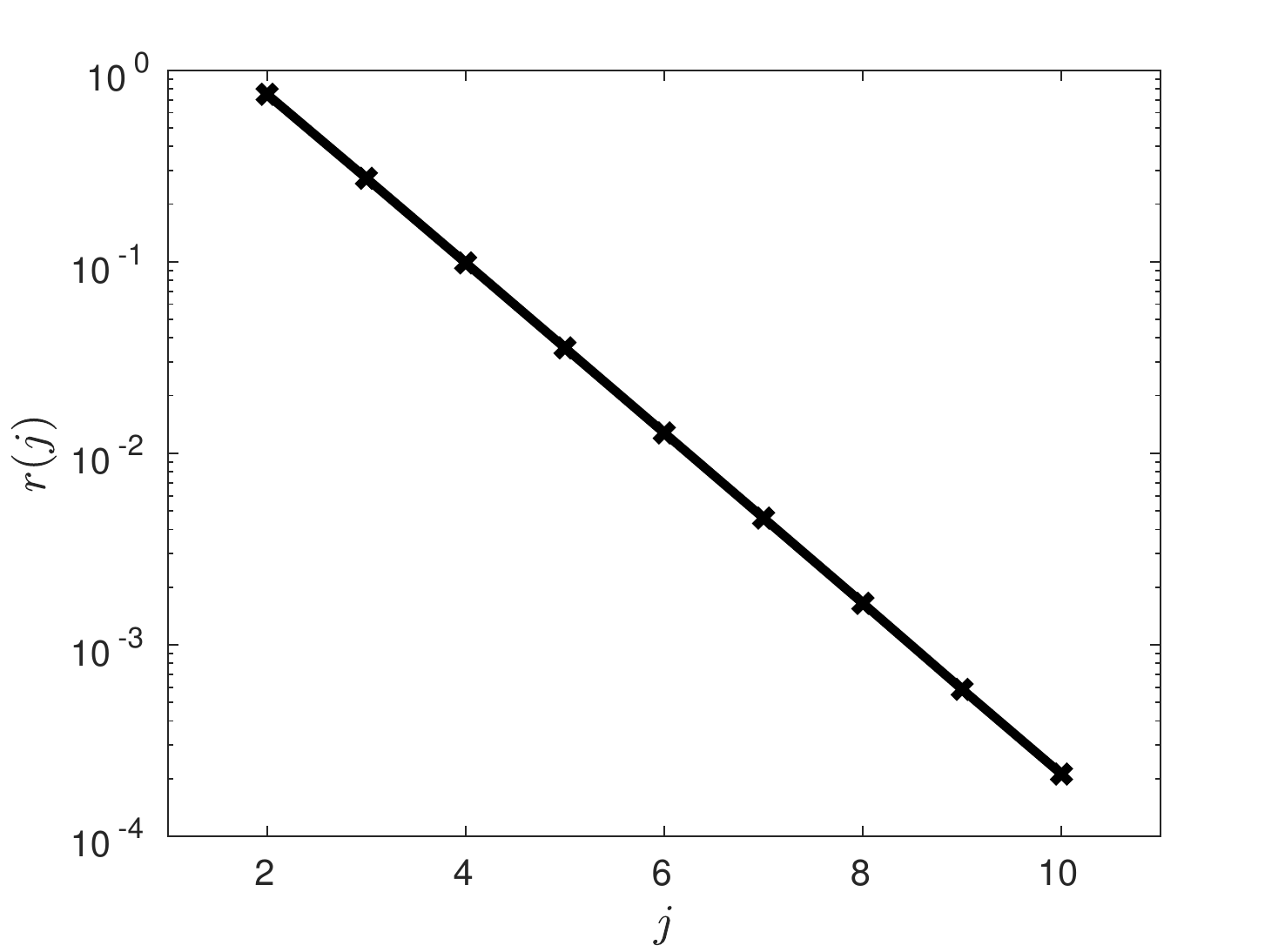} \end{minipage} }
\caption{
For level $j=0:10$. (Left)~Upper and lower bounds for the ground eigenvalue on $\mathsf{H}_{j}$  and $\mathsf{T}_{j}$. The refinement used in each case is as shown in the columns R.  (Right)~Semilog plot of $r(j)$. \label{table2} }
\end{table}

Let the mean of the computed bounds at corresponding level for region $\Omega_j$ be
\[
    \tilde{\omega}_1^2(\Omega_j)=\frac{{\omega}_1^2(\Omega_j)^{\mathrm{upper}}+{\omega}_1^2(\Omega_j)_{\mathrm{lower}}}{2}.
\] On the right of Table~\ref{table2} we show a semilog (vertical axis) plot of 
\[
     r(j)=\tilde{\omega}_1^2(\T_j)-\tilde{\omega}_1^2(\H_j)
\]
versus $j=2:10$. Remarkably, the picture shows a near straight line, suggesting that, to high accuracy, the law 
\[r(j)\approx C\rho^j\]
is satisfied.
Our computed values give $\rho \approx 0.35958$ and
\[
C \approx  5.8688.
\]
See the Remark~\ref{conjecturerate}. In \cite{2007Banjai} this observation was used to accelerate by extrapolation the convergence to the eigenvalues of the fractal.


\appendix

\section{Spectrum of the matrix operator \label{specmat}}
The results presented in this appendix are common knowledge. However, as we could not find a suitable reference to the specific statement that we required in \S\ref{main2}, we include full details of proofs.

Let $\mathfrak{H}_j$ be two possibly different separable Hilbert spaces.
Let $T:\operatorname{D}(T)\longrightarrow \mathfrak{H}_2$ be a densely defined closed operator from  $\operatorname{D}(T)\subset \mathfrak{H}_1$ and let
\[
    \mathcal{E}=\begin{bmatrix} 0 & T^* \\ T & 0 \end{bmatrix}.
\]
The operator 
\[
\mathcal{E}: \operatorname{D}(T)\oplus \operatorname{D}(T^*)\longrightarrow \mathfrak{H}_1\oplus \mathfrak{H}_2
\]
is selfadjoint, indeed note that $(T^*)^*=\overline{T}=T$. Moreover,
by von Neumann's Theorem \cite[p.275]{1967Kato}, we know that both $TT^*:\operatorname{D}(TT^*)\longrightarrow \mathfrak{H}_1$ and  $T^*T:\operatorname{D}(T^*T)\longrightarrow \mathfrak{H}_2$ are selfadjoint in the corresponding domains of operator multiplication (of closed operators). Also $\operatorname{D}(T^*T)\subset \mathfrak{H}_1$ is a core for $T$ and $\operatorname{D}(TT^*)\subset \mathfrak{H}_1$ is a core for $T^*$. As we shall see next, the spectrum of $\mathcal{E}$ is fully characterised by the spectra of $TT^*$ and $T^*T$.  Below, the point spectrum is denoted by $\operatorname{spec}_{\mathrm{p}}$.

\begin{lemma}   \label{zeroinpointspecofEandTTstar}
$0\in \operatorname{spec}_{\mathrm{p}}(\mathcal{E})$ 
if and only if $0\in \operatorname{spec}_{\mathrm{p}}(TT^*)\cup \operatorname{spec}_{\mathrm{p}}(T^*T)$. 
Moreover
\[
   \operatorname{Tr}\mathbb{1}_{0}(\mathcal{E})
   =   \operatorname{Tr}\mathbb{1}_{0}(T^*T)+
   \operatorname{Tr}\mathbb{1}_{0}(TT^*).
\]
\end{lemma}
\begin{proof}
Since
\[
\mathcal{E}\begin{bmatrix} u\\ v \end{bmatrix}=0\ \iff \
T^*v=0  \text{ and } Tu=0 \ \iff \ TT^*v=0  \text{ and } T^*Tu=0,
\]
the first claim follows. 

For the second claim note that there is a one-to-one correspondence between a set of eigenvectors
$\left\{ \begin{bmatrix} u_n \\ \pm v_n \end{bmatrix}\right\}$
associated to $0\in \operatorname{spec}_{\mathrm{p}}(\mathcal{E})$ and
$\left\{ \begin{bmatrix} u_n \\ 0 \end{bmatrix},\begin{bmatrix} 0 \\ v_n \end{bmatrix}\right\}$, which possibly has some zero vectors.
\end{proof}

In the above statement zero can be in the point spectrum of one of the operators $TT^*$ or $T^*T$ but not necessarily the other. This is for example the case for $T$ the standard shift in $\ell^2(\mathbb{N})$.

\begin{lemma}   \label{pointspecofEandTTstar}
Let $\lambda\not=0$. The following are equivalent
\begin{itemize}
\item $\lambda\in \operatorname{spec}_{\mathrm{p}}(\mathcal{E})$ 
\item $-\lambda\in \operatorname{spec}_{\mathrm{p}}(\mathcal{E})$
\item $\lambda^2\in \operatorname{spec}_{\mathrm{p}}(TT^*)$
\item $\lambda^2\in \operatorname{spec}_{\mathrm{p}}(T^*T)$. 
\end{itemize}
Moreover
\[
   \operatorname{Tr}\mathbb{1}_{\lambda}(\mathcal{E})
   =\operatorname{Tr}\mathbb{1}_{-\lambda}(\mathcal{E})=
   \operatorname{Tr}\mathbb{1}_{\lambda^2}(T^*T)=
   \operatorname{Tr}\mathbb{1}_{\lambda^2}(TT^*).
\]
\end{lemma}
\begin{proof}
Let $\lambda\in \operatorname{spec}_{\mathrm{p}}(\mathcal{E})$ and 
$\operatorname{Tr}\mathbb{1}_{\lambda}(\mathcal{E})=m$. Then
there exists a linearly independent set
\[
     \left\{ \begin{bmatrix} u_n \\ v_n \end{bmatrix}\right\}_{n=1}^m\subset\operatorname{D}(\mathcal{E}) \qquad \text{such that}
     \qquad (\mathcal{E}-\lambda) \begin{bmatrix} u_n \\ v_n \end{bmatrix}=0.
\] 
Then $T^*v_n=\lambda u_n$ and $Tu_n=\lambda v_n$ and, necessarily,
$u_n\not=0$ and $v_n\not=0$ for all $n\in\{1,\ldots,m\}$. Thus also the set
$
     \left\{ \begin{bmatrix} u_n \\ - v_n \end{bmatrix}\right\}_{n=1}^m$ is linearly independent 
and
$(\mathcal{E}+\lambda) \begin{bmatrix} u_n \\ - v_n \end{bmatrix}=0.$ Therefore $-\lambda \in \operatorname{spec}_{\mathrm{p}}(\mathcal{E})$ and 
$\operatorname{Tr}\mathbb{1}_{-\lambda}(\mathcal{E})=m$. 

Now, as
\[
    \left\{ \begin{bmatrix} \frac{1}{\lambda}T^* v_n \\ v_n \end{bmatrix}\right\}_{n=1}^m=\left\{ \begin{bmatrix} u_n \\ v_n \end{bmatrix}\right\}_{n=1}^m,
\]
the former is a linearly independent set of eigenvectors with $v_n\in \operatorname{D}(T^*)$ and $T^*v_n\in \operatorname{D}(T)$.
Then $TT^*v_n=\lambda^2v_n$ for the set of non-zero vectors $\{v_n\}_{n=1}^m\subset\operatorname{D}(TT^*)$. Assume that $\operatorname{Tr}\mathbb{1}_{\lambda^2}(TT^*)= l<m$. Then 
\[
     \{v_n\}_{n=1}^m \subset \operatorname{span}\{\tilde{v}_j\}_{j=1}^l
\]
for a linearly independent set $\{\tilde{v}_j\}_{j=1}^l$. Hence
\[
     v_k=\sum_{j=1}^la_j \tilde{v}_j \qquad \text{and} \qquad
     T^*v_k=\sum_{j=1}^la_j T^*\tilde{v}_j
\]
for some $k\in \{1,\ldots,m\}$. Thus
\[
      \begin{bmatrix} \frac{1}{\lambda} T^* v_k \\ v_k \end{bmatrix} = \sum_{j=1}^la_j \begin{bmatrix} \frac{1}{\lambda} T^* \tilde{v}_j \\ \tilde{v}_j \end{bmatrix}
\]
which is a contradiction. Therefore $l\geq m$. But let
$\tilde{u}_j=\frac{1}{\lambda} T^* \tilde{v}_j$ and consider
the set $\left\{ \begin{bmatrix} \tilde{u}_j \\ \tilde{v}_j \end{bmatrix}\right\}_{j=1}^l\subset \operatorname{D}(\mathcal{E})$. This is a linearly independent set of eigenvectors of $\mathcal{E}$ for $\lambda$. This shows that necessarily $l=m$. 

All the above, and a symmetric argument involving $u_n$ instead of $v_n$ and $T^*T$ instead of $TT^*$, are enough to prove the claim. 
\end{proof}

\section*{Acknowledgment}

We thank Michael Levitin for his comments during the preparation of this manuscript. This work was initiated after a talk  by Tom\'{a}\v{s} Vejchodsk\'{y} in Scotland in 2014. This research was completed when the second author was on a study leave at the Czech Technical University in Prague in September-November~2018. Project number \verb+cz.02.2.69/0.0/0.0/16_027/0008465+.

\end{document}